\documentclass[11pt]{article}
\usepackage{graphicx}
\usepackage{subfigure}
\usepackage{epsfig}
\usepackage{amsmath}
\usepackage{amsthm}
\usepackage{amssymb}
\usepackage{lscape}
\usepackage{color}
\usepackage{tikz}
\usepackage{lineno}
\usepackage[ruled,boxed]{algorithm2e}
\usepackage{booktabs}
\usepackage{rotating}

\newtheorem{thm}{Theorem}[section]
\newtheorem{lmm}[thm]{Lemma}
\theoremstyle{definition}
\newtheorem{prop}[thm]{Property}

\newtheorem{rem}{Remark}[section]

\begin{document}

\title{How to Design A Generic Accuracy-Enhancing Filter for Discontinuous Galerkin Methods}%

\author{Xiaozhou Li\footnotemark[1]~\footnotemark[2]}

\renewcommand{\thefootnote}{\fnsymbol{footnote}}
\footnotetext[1]{School of Mathematical Sciences, University of Electronic Science and Technology of China, Chengdu, China. Email: xiaozhouli@uestc.edu.cn. Research supported by NSFC grants 11801062.}
\footnotetext[2]{Corresponding author.}


\date{}

\maketitle

\begin{abstract}
Higher-order accuracy (order of $k+1$ in the $L^2$ norm) is one of the well known beneficial properties of the discontinuous Galerkin (DG) method.  Furthermore, many studies have demonstrated the superconvergence property (order of $2k+1$ in the negative norm) of the semi-discrete DG method.  One can take advantage of this superconvergence property by post-processing techniques to enhance the accuracy of the DG solution.  A popular class of post-processing techniques to raise the convergence rate from order $k+1$ to order $2k+1$ in the $L^2$ norm is the Smoothness-Increasing Accuracy-Conserving (SIAC) filtering.  In addition to enhancing the accuracy, the SIAC filtering also increases the inter-element smoothness of the DG solution.  The SIAC filtering was introduced for the DG method of the linear hyperbolic equation by Cockburn et al. in 2003.  Since then,  there are many generalizations of the SIAC filtering have been proposed.  However, the development of SIAC filtering has never gone beyond the framework of using spline functions (mostly B-splines) to construct the filter function.  In this paper, we first investigate the general basis function (beyond the spline functions) that can be used to construct the SIAC filter.  The studies of the general basis function relax the SIAC filter structure and provide more specific properties, such as extra smoothness, etc.  Secondly, we study the basis functions' distribution and propose a new SIAC filter called compact SIAC filter that significantly reduces the original SIAC filter's support size while preserving (or even improving) its ability to enhance the accuracy of the DG solution.  We show that the proofs of the new SIAC filters' ability to extract the superconvergence and provide numerical results to confirm the theoretical results and demonstrate the new finding's good numerical performance.          

\end{abstract}

\section{Introduction}
In the past decades, the DG method has become a popular class of numerical methods for solving partial differential equations, such as hyperbolic equations, convection-diffusion equations, etc.   The first introduction of the DG method was given by Reed and Hill~\cite{Reed:1973} in 1973. Later, this method was extended to the framework of explicit RKDG method for solving time-dependent hyperbolic conservation laws by Cockburn, et al.~\cite{Cockburn:1990, Cockburn:1989b,  Cockburn:1989a, Cockburn:1991}.  It is well known that the DG method has many advantages, such as the flexibility of without a global continuity requirement and high order accuracy for smooth solutions, etc.  Furthermore, many studies have demonstrated the superconvergence property (order of $2k+1$ in the negative norm) of the semi-discrete DG methods~\cite{Cockburn:2003}.  Post-processing techniques can be used to enhance the accuracy of the DG solution due to this superconvergence property.    

A popular class of post-processing techniques to enhance the accuracy of DG solutions is the Smoothness-Increasing Accuracy-Conserving (SIAC) filtering.  The SIAC filtering was originally introduced by Bramble and Schatz~\cite{Bramble:1977} in 1977 for enhancing the accuracy of solutions of finite element methods. Cockburn et al. gave the first extension of this post-processing technique to DG methods.~\cite{Cockburn:2003}.  After applying the SIAC filtering, in an ideal situation, the post-processing technique can increase the accuracy order of DG solutions (with polynomials of degree $k$) from $k+1$ to $2k+1$ in the $L^2$ norm.  In addition to increasing the accuracy order, the SIAC filtering also raises the DG solution's inter-element continuity to $\mathcal{C}^{k-1}$.  This technique has attracted increased attention in recent years due to these good features and its cheap computational cost.  Since the introduction of the SIAC filtering for the DG method, many generalizations of SIAC filtering have been proposed from various perspectives.  To name a few, it has been extended to the boundary (position-dependent) filtering~\cite{Ryan:2015, Ryan:2005, vanSlingerland:2011}, the derivative filtering~\cite{Li:2016, Ryan:2009} as well as the extension to nonuniform meshes~\cite{Curtis:2008, Li:2019}.  It also has shown useful for applications in visualisation~\cite{Steffen:2008, Walfisch:2009}, shock capturing~\cite{Bohm:2019, Wissink:2018}, etc.  Hence, the B-spline based SIAC filtering techniques become one of the most well-accepted filtering techniques to enhance the DG solution's smoothness and accuracy.

However, the development of SIAC filtering has never gone beyond the usage of spline functions.  Most literature had used only B-splines (mostly, central B-splines) to construct the filter, and in a more recent research~\cite{Mirzargar:2017} had used the hex-splines (a generalization of two-dimensional B-splines to hexagonal lattice).  This paper will first investigate the generic basis function (beyond the spline functions) that can be used to construct the SIAC filter while preserving the desired ability to extract the $2k+1$ superconvergence.  The studies of the general basis function relax the structure of the SIAC filter and can provide more specific properties, such as extra the smoothness, etc.  We prove that the superconvergence property holds for this generic basis functions based SIAC filter and provides numerical results to demonstrate and validate our theoretical results.

Secondly, we study the basis functions' distribution of the filter.  The basis functions' distribution has never been studied and changed since its introduction in 2003~\cite{Cockburn:2003}.  By investigating the way of distribution, we propose a new SIAC filter called compact SIAC filter that significantly reduces the support size of the original SIAC filter.  The original SIAC filter has a support size of $3k+1$, which is quite large for higher-order $k$, especially for the multi-dimensional case.  Besides increasing the computational cost, a large support size may be problematic when introducing boundary conditions, when the mesh is highly unstructured or when the solution is lack of smoothness, etc.  The newly proposed compact SIAC filter can reduce the support size from $3k+1$ to $k+2$ and maintain the same (or even better) accuracy-enhancing ability of the original SIAC filter.  We also prove that the compact SIAC filter has the same superconvergence property as the original SIAC filter, and demonstrates better numerical performance.

The paper proceeds as follows. We first introduce the used notation in Section~\ref{sec:notation}, then we provide a brief review of the DG method in Section~\ref{sec:dg} and the standard SIAC filtering with central B-splines in Section~\ref{sec:siac}.  In Section~\ref{sec:basis}, we investigate the generic basis functions to construct the SIAC filter.  We introduce the compact SIAC filter and its superconvergence property as well as some numerical results in Section~\ref{sec:compact}.  Finally, we present the conclusion in Section~\ref{sec:conclusion}.

\section{Background}
\subsection{Notation}
\label{sec:notation}
We start by introducing the norms of function spaces used in the remainder of the paper. 
Consider a domain $\Omega \subset \mathbb{R}^d$,  one can define the standard $L^2$-norm over $\Omega$ as
\[ \|u\|_{0,\Omega} = \left(\int_\Omega u^2 dx\right)^{\frac{1}{2}}.\]
For any non-negative integer $\ell$,  the norm of Sobolev space $H^\ell(\Omega)$ is given by
\[ 
    \|u\|_{\ell,\Omega} = \left(\sum\limits_{|\alpha|\leq \ell}\| D^\alpha u\|_{0,\Omega}^2\right)^{\frac{1}{2}},
\]
where $D^\alpha$ denotes the differentiation operator of degree $\alpha$.  Further, we denote the negative norm on the domain $\Omega$ (or say the norm of the dual space of $H^\ell(\Omega)$) as 
\[ \|u\|_{-\ell,\Omega} = \sup\limits_{\phi \in \mathcal{C}^\infty_0(\Omega)}\frac{(u, \phi)_\Omega}{\|\phi\|_{\ell,\Omega}},\]
where $(\cdot,\cdot)$ represents the inner product of $H^\ell(\Omega)$.  

At last, we introduce the divided differences operator, which is crucial for the superconvergence extracting technique discussed in the rest of the paper.  For the one-dimension case,
\[
    \partial^\alpha_h u(x) = \left\{\begin{array}{cc}
    \frac{1}{h}\left(u\left(x+\frac{h}{2}\right) - u\left(x+\frac{h}{2}\right)\right),  & \alpha = 1; \\
    \partial_h\left(\partial_h^{\alpha-1} u\right), & \alpha > 1.\end{array}\right.
\]
The multi-dimensional version can be defined analogously. 

\subsection{The DG Method}
\label{sec:dg}
In this section, we give a brief description of the essential concepts of the DG method.  One can refer to~\cite{Cockburn:1990, Cockburn:1989b,  Cockburn:1989a, Cockburn:1991} for more details of the DG method.  

Consider the one-dimensional linear hyperbolic equation
\begin{equation}
  \label{eq:hyperbolic}
  u_t + au_x = 0, \qquad (x,t) \in \Omega\times[0,T],
\end{equation}
where the spatial domain $\Omega = [a, b] \subset \mathbb{R}$.  Throughout this paper, we adopt the periodic boundary conditions and a sufficient smooth initial condition
\[
u(x, 0) = u_0(x) \in \mathcal{C}^\infty(\Omega).
\]
To introduce the DG scheme, we first give a partition of $\Omega$ that
\[
    a = x_{\frac{1}{2}} < x_{\frac{3}{2}} < \cdots < x_{N+\frac{1}{2}} = b,
\]
and denote the element $I_j = [x_{j-\frac{1}{2}}, x_{j+\frac{1}{2}}]$ with size $\Delta x_j = x_{j+\frac{1}{2}} -  x_{j-\frac{1}{2}}$ for $j = 1,\dots,N$.  For convenience, in this paper, we limit the discussion to the uniform partition (it does not need to be) that $\Delta x_j = h = \frac{1}{N}(b-a)$ for $j = 1,\dots,N$.  Associated with the partition, we can define the DG approximation space
\begin{equation}
  \label{eq:space-s}
  V_h  = \left\{v_h:\, v_h\big\vert_{I_j} \in \mathbb{P}^k(I_j),\,j = 1, \dots, N\right\},
\end{equation}
where $\mathbb{P}^k(I_j)$ denotes the set of polynomials of degree up to $k$ defined on the element $I_j$.  Clearly, $V_h$ is a piecewise polynomial space, and functions in $V_h$ may have discontinuities across element interfaces.  Therefore, we need denote the notation $u_h(x_{j-\frac{1}{2}}^{-})$ and $u_h(x_{j-\frac{1}{2}}^{+})$ for $u_h \in V_h$ i.e.\ the value of $u_h$ at $x_{j-\frac{1}{2}}$ from the left side and from the right side, respectively.   

Then, the DG solution, $u_h$, for equation~\eqref{eq:hyperbolic} satisfies the semi-discretized weak formulation
\begin{equation}
  \label{eq:weak-s}
  \int_{I_j} {(u_h)}_t v_h \,dx - \int_{I_j}a u_h {(v_h)}_x \,dx + \hat{f}v_h\vert_{x_{j+\frac{1}{2}}^{-}} - \hat{f}v_h\vert_{x_{j-\frac{1}{2}}^{+}} = 0, \quad \forall v_h \in V_h,
\end{equation}
for all $1 \leq j \leq N$.  Here, the ``hat'' terms $\hat{f} = \hat{f}(u_h^{-},u_h^{+})$ is referred as the numerical flux, which is essential to ensure the stability of the numerical scheme.  In this paper, we simply use the upwind flux 
\[
  \hat{f} = \hat{f}\left(u_h^{-},u_h^{+}\right) = a u_h^{-}.
\]
For the error estimate in the $L^2$ norm, it is well-known that the exact solution $u(\cdot)$ for equation~\eqref{eq:hyperbolic} and the DG approximation $u_h(\cdot)$ for the weak formulation~\eqref{eq:weak-s} satisfy 
\[
    \|u - u_h \|_{0,\Omega} \leq C h^{k+1}.
\]
Another crucial feature of the DG method is the so-call superconvergence property. The superconvergence of the DG method is a phenomenon where the order of convergence, under certain norms, is higher than the accuracy order under the $L^2$ norm.  The focus of this paper, the accuracy-enhancing filtering, is developed mainly on the superconvergence property of the DG approximation and its divided differences under the negative norm.  For uniform meshes, we give the main theorem as follows.
\begin{thm}[Cockburn et al.~\cite{Cockburn:2003}]
    \label{thm-0}
    Let $u$ be the exact solution of equation~\eqref{eq:hyperbolic} with periodic boundary conditions, and $u_h$ the DG approximation derived by scheme~\eqref{eq:weak-s}. For a uniform mesh, the approximation and its divided differences in the $L^2$ norm, we have the following error estimate: 
\begin{equation}
    \label{eq-L2}
    \|\partial_h^{\alpha} (u-u_h)\|_{0,\Omega} \leq  C h^{k+1},
\end{equation}
and in the negative norm:
\begin{equation}
    \label{eq-neg}
    \|\partial_h^{\alpha} (u-u_h)\|_{-(k+1),\Omega} \leq  C h^{2k+1},
\end{equation}
where $\alpha$ is an arbitrary non-negative integer. 
\end{thm}
The relation between the $L^2$ norm and the negative norm was given by
\begin{lmm}[Bramble and Schatz \cite{Bramble:1977}]
    \label{lmm-bramble}
    Let $\Omega_0 \subset\subset \Omega_1$ and $s$ be an arbitrary but fixed non-negative integer. Then for $u \in H^s(\Omega_1)$, there exist a constant $C$ such that
    \[ \|u\|_{0,\Omega_0} \leq C \sum\limits_{|\alpha|\leq s}\|D^\alpha u\|_{-s,\Omega_1}. \]
\end{lmm}
Theorem~\ref{thm-0} is the superconvergence property of the DG method, and Lemma~\ref{lmm-bramble} shows that the negative norm of a function and its derivatives can bound the $L^2$ norm of the function itself.  Together, they construct the theoretical foundation of techniques for extracting the superconvergence of the DG solution.  Those techniques are usually referred to as superconvergence extract or post-processing techniques.  One popular class of these techniques is the so-called smoothness-increasing accuracy-conserving (SIAC) filtering.  The purpose of this paper is to investigate how to design a generic accuracy-enhancing filter for the DG method.  Before further discussion, we first review the standard SIAC filtering and the properties for constructing accuracy-enhancing filtering.

\subsection{The SIAC Filtering}
\label{sec:siac}
Along with developing the DG method, the SIAC filtering study has become an area of increased interest.  The primary attracted attention of the SIAC filtering is its ability to extract the superconvergence (higher-order information) from the DG solution to achieve the purpose of enhancing the accuracy.  The SIAC filtering sourced from the post-processing technique of Bramble and Schatz~\cite{Bramble:1977} proposed for the finite element method, and extended to the DG method by Cockburn et al.~\cite{Cockburn:2003} in 2003. Since then, many researchers have studied various aspects of the SIAC filtering techniques in the literature, see~\cite{Ji:2013, Li:2016, Li:2015, Ryan:2015, Ryan:2009}.  In this section, we review the basic structure of the SIAC filtering as follows. 

The SIAC filtering is applied only at the final time, $T$ of the DG solution.  The filtered solution, $u_h^{\star}$, is given by 
\begin{equation}
  \label{eq:ustar}
    u_h^{\star}(x, T) = \left(K_H^{(2k+1,\,k+1)}\ast u_h\right)(x, T)= \int_{-\infty}^{\infty}K_H^{(2k+1,\,k+1)}(x - \xi)u_h(\xi, T)\, d\xi.
\end{equation}
In the filtering convolution~\eqref{eq:ustar}, the SIAC filter, $K^{(2k+1,\,k+1)}$, is a linear combination of certain basis functions $\psi^{(k+1)}$,
\begin{equation}
    \label{eq:filter}
    K^{(2k+1,\,k+1)}(x) = \sum\limits_{\gamma=0}^{2k}c^{(2k+1,\,k+1)}_{\gamma}\psi^{(k+1)}\left(x - x_\gamma\right),
\end{equation}
where 
\begin{equation}
    \label{eq:cpoints}
    x_\gamma = -k + \gamma
\end{equation} 
denotes the central point of the $\gamma-$th basis function.  Moreover, the scaled filter, $K_H$, is given by
\[
  K_H^{(2k+1,\,k+1)}(x) = \frac{1}{H}K^{(2k+1,\,k+1)}\left(\frac{x}{H}\right)
\] 
where $H$ is the filter scaling.  For uniform meshes, we usually choose the scaling to be the (uniform) element size $H = h$, and denote the scaled filter $K_H$ as $K_h$.  For non-uniform meshes, the simple choice is choosing the scaling as the local element size~\cite{Ryan:2015} or the largest element size~\cite{Curtis:2008}.  For more recent research results, one can determine the scaling according to the structure of the given non-uniform mesh to obtain the optimal accuracy, see~\cite{Li:2019}.

\subsubsection{The Central B-spline}
The basis functions used in the SIAC filter~\eqref{eq:filter} are central B-splines.  The $k+1$ order central B-spline,$\,\psi^{(k+1)}(x)$, can be constructed recursively
\begin{equation}
    \begin{split}
      \psi^{(1)} & = \chi_{[-\frac{1}{2},\frac{1}{2}]}(x), \\
        \psi^{(\ell+1)}(x) & =  \frac{1}{\ell}\left(\left(\frac{\ell+1}{2}+x\right)\psi^{(\ell)}\left(x+\frac{1}{2}\right) + \left(\frac{\ell+1}{2}-x\right)\psi^{(\ell)}\left(x-\frac{1}{2}\right)\right), \,\, \text{for~} \ell \geq 1.
    \end{split}
\end{equation}
For example, we give the analytic formulas and plots (see Figure~\ref{figure-central-bspline}) of the central B-splines $\psi^{(k+1)}$ (with $k = 1, 2, 3$) as follows:
\begin{equation}
    \label{eq:bsplines}
  \begin{split}
    \psi^{(2)}(x) & = \left\{\begin{array}{ll}
        1 + x, & x \in [-1,0), \\
        1 - x, & x \in [0, 1], \\
        0, & \text{otherwise};
    \end{array} \right. \\
    \psi^{(3)}(x) & = \left\{\begin{array}{ll}
        \frac{1}{8}(2x+3)^2, & x \in [-\frac{3}{2},-\frac{1}{2}), \\
        \frac{1}{4}(-4x^2 + 3), & x \in [-\frac{1}{2}, \frac{1}{2}), \\
        \frac{1}{8}(2x-3)^2, & x \in [\frac{1}{2},-\frac{3}{2}] \\
        0, & \text{otherwise};
    \end{array} \right. \\
    \psi^{(4)}(x) & = \left\{\begin{array}{ll}
        \frac{1}{6}(x+2)^3,        & x \in [-2,-1), \\
        \frac{1}{6}(-3x^3-6x^2+4), & x \in [-1, 0), \\
        \frac{1}{6}( 3x^3-6x^2+4), & x \in [ 0, 1), \\
        \frac{1}{6}(2-x)^3,        & x \in [ 1, 2], \\
        0, & \text{otherwise}.
    \end{array} \right.
  \end{split}
\end{equation}

\begin{figure}[!ht]
    \centering
    \begin{minipage}[c]{0.32\textwidth}
        \centerline{$\quad k = 1$}
    \end{minipage}
    \begin{minipage}[c]{0.32\textwidth}
        \centerline{$\quad k = 2$}
    \end{minipage}
    \begin{minipage}[c]{0.32\textwidth}
        \centerline{$\quad k = 3$}
    \end{minipage}

    \begin{minipage}[c]{0.32\textwidth}
        \includegraphics[width=1.\textwidth]{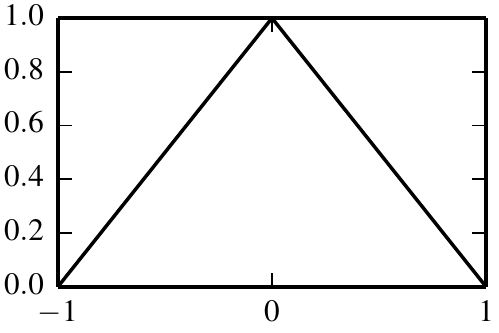}
    \end{minipage}
    \begin{minipage}[c]{0.32\textwidth}
        \includegraphics[width=1.\textwidth]{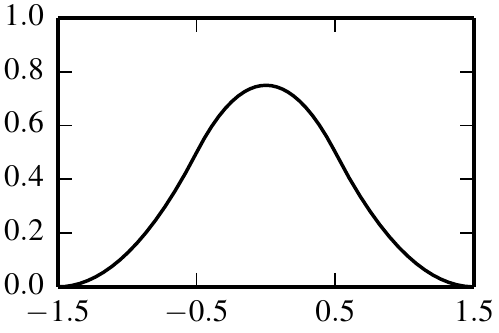}
    \end{minipage}
    \begin{minipage}[c]{0.32\textwidth}
        \includegraphics[width=1.\textwidth]{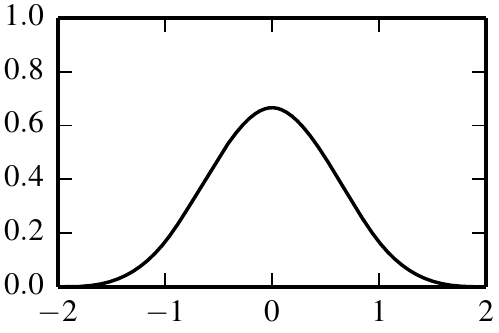}
    \end{minipage}

    \begin{center}
    \caption{\label{figure-central-bspline}
        Central B-spline $\psi^{(k+1)}$ with $k = 1, 2, 3$, see~\eqref{eq:bsplines}. 
}
     \end{center}
\end{figure}
The main reason for using the central B-spline to construct the filter is that its special properties that aid in the proofs of extracting higher-order accuracy in the negative norm.  The crucial one of these properties is differentiation property:
\begin{prop}[Differentiation of Central B-spline]
    \label{prop-centralspline}
    The $\alpha$th derivative of a central B-spline is given by 
    \[ 
        D^\alpha \psi_h^{(\ell)} = \partial^\alpha_h \psi_h^{(\ell-\alpha)}, 
    \]
    where $\psi_h^{(\ell)}$ is the central B-spline with scaling $h$. 
\end{prop}
Property~\ref{prop-centralspline} shows that the derivatives of a central B-spline can be expressed only by its divided differences, and it directly leads to a similar property of the filter $K^{(2k+1,k+1)}$ as the following Property~\ref{prop-divided}. 
\begin{prop}
    \label{prop-divided}
    As a consequence of the filter constructed by using central B-splines (Property~\ref{prop-centralspline}), one can express derivatives of the convolution with the filter in terms of simple difference quotients.  It is trivial to verify that 
    \[ 
        D^\alpha(K_h^{(2k+1,k+1)} \star u_h) = \tilde{K}_h^{(2k+1,k+1-\alpha,\alpha)} \star \partial_h^\alpha u_h, 
    \]
    where $\alpha$ is a non-negative integer ($\alpha \leq k+1$) and 
    \[
        \tilde{K}_h^{(2k+1,k+1-\alpha,\alpha)} = \sum\limits_{\gamma=0}^{2k} c_{\gamma}^{(2k+1,k+1)}\psi^{(k+1-\alpha)}(x - x_\gamma).
    \]
\end{prop}

The coefficients $c^{(2k+1,k+1)}_{\gamma}$ in~\eqref{eq:filter} are decided by requiring the filter satisfies the property (Property~\ref{prop:reprod}) that reproduces polynomials by convolution up to degree $2k$, 
\begin{prop}
    \label{prop:reprod}
    The filter $K^{(2k+1,k+1)}$ defined in~\eqref{eq:filter} satisfies the property of reproducing polynomial by convolution until degree of $2k$,
    \begin{equation}
        \label{eq:reprod}
        K^{(2k+1,k+1)} \star p = p,\, p = 1, x, \ldots, x^{2k}.
    \end{equation}
\end{prop}
Now, we have constructed the SIAC filter by using central B-splines.  In Figure~\ref{figure-filter}, we present the plots of the filter $K^{(2k+1,k+1)}$ with $k = 1, 2 ,3$.

\begin{figure}[!ht]
    \centering
    \begin{minipage}[c]{0.32\textwidth}
        \centerline{$\quad k = 1$}
    \end{minipage}
    \begin{minipage}[c]{0.32\textwidth}
        \centerline{$\quad k = 2$}
    \end{minipage}
    \begin{minipage}[c]{0.32\textwidth}
        \centerline{$\quad k = 3$}
    \end{minipage}

    \begin{minipage}[c]{0.32\textwidth}
        \includegraphics[width=1.\textwidth]{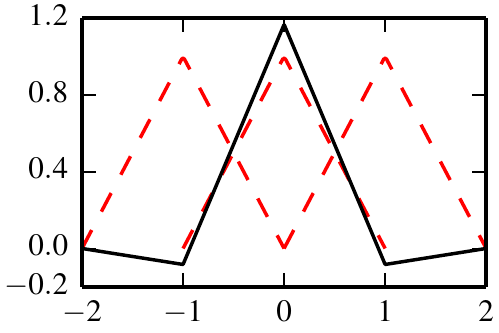}
    \end{minipage}
    \begin{minipage}[c]{0.32\textwidth}
        \includegraphics[width=1.\textwidth]{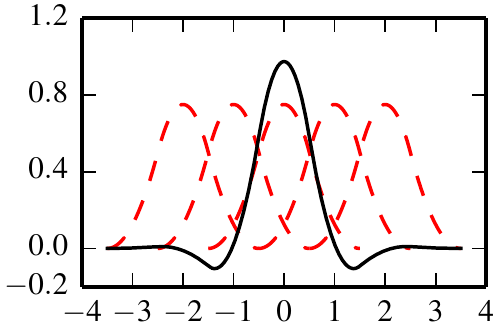}
    \end{minipage}
    \begin{minipage}[c]{0.32\textwidth}
        \includegraphics[width=1.\textwidth]{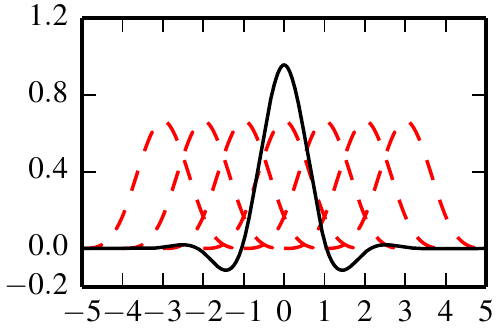}
    \end{minipage}

    \begin{center}
    \caption{\label{figure-filter}
        Solid black lines represent the SIAC filter $K^{(2k+1,k+1)}(x)$ with $k = 1, 2, 3$, dashed red lines represent the respective central B-splines.  The filtered point is located at $x = 0$.
}
     \end{center}
\end{figure}

Properties~\ref{prop-divided} and~\ref{prop:reprod} are the key to extract the superconvergence from DG solutions, together with Theorem~\ref{thm-0} and Lemma~\ref{lmm-bramble} we obtain the error estimates for the filtered solution $u_h^\star$.
\begin{thm}[Cockburn et al.~\cite{Cockburn:2003}]
    \label{thm-1}
    Under the same conditions in Theorem~\ref{thm-0}, denote $\Omega_0 + 2 \text{supp}(K_h^{(2k+1,k+1)}) \subset\subset \Omega_1 \subset \Omega$, then 
    \[ 
        \|u - K_h^{(2k+1,k+1)}\star u_h\|_{0,\Omega_0} \leq C h^{2k+1}.
    \]
\end{thm}    

Furthermore, besides Properties~\ref{prop-divided} and~\ref{prop:reprod}, we summarize other features of the SIAC filter $K^{(2k+1,k+1)}$ that will be discussed in the rest of the paper:
\begin{itemize}
    \item Compact support, the support size is $3k+1$;
    \item Symmetry with respect to the filtered point ($x=0$);
    \item The filter is a $\mathcal{C}^{k-1}$ function and therefore so is the filtered solution $u_h^\star$.
\end{itemize}

\section{The Basis Function of SIAC Filter}
\label{sec:basis}
In the previous section, we reviewed the SIAC filter's structure - a linear combination of central B-splines.  Based on this B-spline based structure, many generalizations of SIAC filtering have been proposed, such as~\cite{Ji:2013, Li:2016, Li:2015, Ryan:2015, Ryan:2009}.  However, the SIAC filter development scope has never gone beyond the usage of spline functions (mostly B-splines).  Naturally, there is an interesting question: are there other basis functions for the SIAC filter that still allow one to extract the superconvergence order of $2k+1$ from DG solutions, or what is the generic structure of the basis functions for an accuracy-enhancing filter? 

To answer this question, we first recall the necessary components to extract the superconvergence order of $2k+1$, Theorem~\ref{thm-1}.  By checking the proof of Theorem~\ref{thm-1} given in~\cite{Cockburn:2003}, one can see that there are four main components behind Theorem~\ref{thm-1}, namely, Theorem~\ref{thm-0}, Lemma~\ref{lmm-bramble}, Property~\ref{prop-divided} and Property~\ref{prop:reprod}.  Among these four components, Theorem~\ref{thm-1} is the superconvergence property of the DG method, and Lemma~\ref{lmm-bramble} is the property of Hilbert space.  Therefore, to construct an accuracy-enhancing filter, we need to focus on the rest two components that
\begin{itemize}
    \item Property~\ref{prop-divided} allows us to express the filter's derivatives in terms of divided difference quotients.  It is from the structure of the basis function, more precisely, the central B-splines (Property~\ref{prop-centralspline}).
    \item Property~\ref{prop:reprod} demonstrates that the filter can reproduce polynomials by convolution.  It is due to equation~\eqref{eq:reprod}.
\end{itemize}
Through the above analysis, we know the filter's fundamental structure is to provide the ability to express its derivatives in terms of divided differences. Therefore, the first task is to find a basis function $\phi^{(\ell)}$, such that
\begin{equation}
    \label{eq:divided} 
    \frac{d}{dx}\phi^{(\ell)}_h(x) = \partial_h\phi^{(\ell-1)}(x) =\frac{1}{h}\left(\phi^{(\ell-1)}\left(x + \frac{h}{2}\right) - \phi^{(\ell-1)}\left(x - \frac{h}{2}\right)\right).
\end{equation}
To construct a series basis functions $\left\{\phi^{(k)}\right\}$ satisfy relation~\eqref{eq:divided} , we take the Fourier transform of equation~\eqref{eq:divided}
\[
    2\pi i\xi \hat{\phi}^{(\ell)} = \hat{\phi}^{(\ell-1)}e^{\pi i \xi} - \hat{\phi}^{(\ell-1)}e^{-\pi i \xi}.
\]
The above Fourier form leads to
\begin{equation}
    \label{eq:fourier}
    \hat{\phi}^{(\ell)}  = \hat{\phi}^{(\ell-1)}\frac{\sin(\pi\xi)}{\pi\xi}.
\end{equation}
Also, it is easy to verify that the Fourier transform of the characteristic function $\hat{\chi}_{[-1/2,1/2]}$ satisfies
\[
    \hat{\chi}_{[-1/2,1/2]} = \frac{\sin(\pi\xi)}{\pi\xi}.
\]
Therefore, equation~\eqref{eq:fourier} becomes
\[
    \hat{\phi}^{(\ell)}  = \hat{\phi}^{(\ell-1)}\hat{\chi}_{[-1/2,1/2]}.
\]
By taking the inverse Fourier transform of the above equation, we have
\begin{equation}
    \phi^{(\ell)} = \phi^{(\ell-1)}\star \chi_{[-1/2,1/2]}.
\end{equation}
That is, once the initial basis function $\phi^{(1)}$ is chosen, one can construct a series basis functions $\left\{\phi^{(k)}\right\}$ recursively
\begin{equation}
    \label{eq:recursion}
    \begin{split}
      \phi^{(1)} & = \text{initial basis function}, \\
      \phi^{(\ell+1)}(x) & =  \phi^{(\ell)}\star \chi_{[-1/2,1/2]},\,\, \text{for~} \ell \geq 1.
    \end{split}
\end{equation}
Once the basis functions are decided, we can construct a generic accuracy-enhancing filter, that is 
\begin{equation}
    \label{eq:general_filter}
    K^{(2k+1,k+1)} = \sum\limits_{\gamma=0}^{2k} c_\gamma^{(2k+1,k+1)}\phi^{(k+1)}(x - x_\gamma),
\end{equation}
where $c_\gamma^{(2k+1,k+1)}$ are decided by requiring the filter satisfies Property~\ref{prop:reprod} the same way as the original SIAC filter~\eqref{eq:filter}.  That is, the generic accuracy-enhancing filter also needs the ability to reproduce polynomials by convolution until a degree of up to $2k$.  To show the existence and uniqueness of such coefficients $c^{(2k+1,k+1)}_\gamma$, we present the following Theorem~\ref{thm:coeff} with a more general assumption.

\begin{thm}
    \label{thm:coeff}
    Assume $\phi_\gamma,\,\gamma = 0,\ldots,r$ are $r+1$ linear independent functions, each has a compact support and satisfies $\int_{-\infty}^{\infty} \phi_\gamma(\xi)\,d\xi \neq 0$.  Then the linear system 
    \begin{equation}
        \label{eq:repo}
        \sum\limits_{\gamma = 0}^r c_\gamma \int_{-\infty}^\infty \phi_\gamma(\xi)(x - \xi)^m d\xi = x^m,\quad m = 0,1,\ldots,r
    \end{equation}
has a unique solution.
\end{thm}
\begin{proof}
    Without loss of generality, we assume $\int_{-\infty}^{\infty} \phi_\gamma(\xi)\,d\xi=1$, and denote  
    \[ 
        \begin{split}
            A(\gamma, m)(x) & = \int_{-\infty}^\infty \phi_\gamma(\xi)(x - \xi)^m d\xi,\quad \gamma, m = 0,1,\ldots, r, \\
                            & = x^m + \sum\limits_{j=0}^{m-1}\lambda_{\gamma,m}^j x^j,
        \end{split}
    \]
    where $\lambda_{\gamma,m}^j = \int_{-\infty}^\infty \phi_\gamma(\xi)\left(-\xi\right)^j \binom{m}{j} d\xi$.  Then, the conclusion in Theorem \ref{thm:coeff} is equivalent to matrix $A(\gamma, m)$ is non-singular, in other words, we only need to prove the $r+1$ rows $\left\{A(\gamma, m)\right\}^r_{\gamma=0}$ are linear independent.

    Assume there exist constants $b_m, m = 0,\ldots, r$ such that
    \[ 
        \sum\limits_{m = 0}^r b_m A(\gamma,m)(x) = 0, \quad \gamma = 0,\ldots, r.
    \]
    By substituting the form of $A(\gamma,m)$, we have 
    \[
        \sum\limits_{m = 0}^r b_m \left(x^m + \sum\limits_{j = 0}^{m-1}\lambda_{\gamma,m}^j x^j\right) = 0 \, \Rightarrow \, \sum\limits_{m=0}^r\left(b_m + \sum\limits_{j=m+1}^r b_j \lambda_{\gamma,i}^m\right) x^m = 0. 
    \]
    Since the above relation is valid for all $x \in \mathbb{R}$, we have
    \[
        b_m + \sum\limits_{j=m+1}^r b_j \lambda_{\gamma,i}^m = 0 \, \Rightarrow \, \left\{\begin{array}{l}
                        b_r = 0, \\
                        b_{r-1} = -\lambda_{\gamma,r}^{r-1} b_r = 0, \\
                        \ldots \\
                        b_0 = -\left(\sum\limits_{j=1}^r \lambda_{\gamma,j}^0 b_j\right) = 0.
                    \end{array}\right.
    \]
    It follows that the $r+1$ rows of linear system~\eqref{eq:reprod} is linear independent. 
\end{proof}
Since $\phi^{(k+1)}(x - x_\gamma), \gamma = 0,\ldots,2k$, are $2k+1$ linear independent compact functions, Theorem~\ref{thm:coeff} indicates that there exists unique constants $c^{(2k+1,k+1)}_\gamma$ such that filter~\eqref{eq:general_filter} reproduces polynomials by convolution up to degree $2k$.

Finally, we find a way to construct a general basis function for the accuracy-enhancing filter~\eqref{eq:general_filter}, which also satisfies Property~\ref{prop-divided} and Property~\ref{prop:reprod}.  We know that there are infinitely many basis functions for the SIAC filter that allow us to extract the superconvergence order of $2k+1$ from the DG solutions.  
\begin{thm}
    \label{thm:general}
    Under the same conditions in Theorem~\ref{thm-1}, and the filter given in~\eqref{eq:general_filter}, we have 
    \[ 
        \|u - K_h^{(2k+1,k+1)}\star u_h\|_{0,\Omega_0} \leq C h^{2k+1}.
    \]
\end{thm}
\begin{proof}
    Here, we only give a sketch proof of the error estimate of the filtered solution, and one can refer the proof of Theorem~\ref{thm-1} for more details.  
    
    We first isolate components of the error estimate by using the triangle inequality.  That is,
    \[
        \left\|u-K^{(2k+1,k+1)}_h\star u_h\right\|_0\leq\underbrace{\left\|u-K^{(2k+1,k+1)}_h\star u\right\|_{0}}_{\text{filter construction}}+\underbrace{\left\|K^{(2k+1,k+1 )}_h\star(u-u_h)\right\|_{0}}_{\text{approximation error}}
    \]
    The first term is solely determined from Property~\ref{prop:reprod} of the filter.  By taking the Taylor polynomial of $u$ and Property~\ref{prop:reprod}, we can prove 
    \[
        \left\|u-K^{(2k+1,k+1)}_h\star u\right\|_{0} \leq C h^{2k+1}.
    \]
    The second term relies on the approximation error, by Lemma~\ref{lmm-bramble}, 
    \[
        \left\|K^{(2k+1,k+1)}_h\star(u-u_h)\right\|_{0} \leq C\sum\limits_{\alpha \leq k+1}\left\|D^\alpha\left(K^{(2k+1,k+1)}_h\star(u-u_h)\right)\right\|_0.
    \]
    Due to the recursion relation of the basis function, the general accuracy-enhancing filter~\eqref{eq:general_filter} also satisfies Property~\ref{prop-divided}, then
    \[
        \left\|D^\alpha\left(K^{(2k+1,k+1)}_h\star(u-u_h)\right)\right\|_0 = \left\|\left(\tilde{K}^{(2k+1,k+1-\alpha,\alpha)}_h\star \partial_h^\alpha (u-u_h)\right)\right\|_0.
    \]
    By combining with the superconvergence property of the DG method (Theorem~\ref{thm-0}), we have
    \[
        \left\|K^{(2k+1,k+1)}_h\star(u-u_h)\right\|_{0} \leq C h^{2k+1}.
    \]
    That is, 
    \[ 
        \left\|u - K_h^{(2k+1,k+1)}\star u_h\right\|_{0} \leq C h^{2k+1}.
    \]
\end{proof}

\subsection{The Raised Cosine Basis}
In the previous section, we demonstrate that once the initial basis function $\phi^{(1)}$ is decided, the filter $K^{(2k+1,k+1)}$ constructed by~\eqref{eq:general_filter} with the generic basis function $\phi^{(k+1)}$~\eqref{eq:recursion}.  We also proved that this filter still preserves the ability of the central B-spline filter that to raise the convergence rate of the DG solution from order $k+1$ to $2k+1$ (Theorem~\ref{thm:general}).  We note that the natural choice is $\phi^{(1)} = \chi_{[-1/2,1/2]}$, which leads to the central B-spline and the original SIAC filter. If we limit the SIAC filter to be a piecewise polynomial, then the central B-spline filter has the simplest formula (or lowest degree).  However, there are many choices with different features, such as the filters used in the spectral method~\cite{Vandeven:1991}, the Dirac-Delta filter used in~\cite{Bohm:2019, Wissink:2018}, etc.  

In this section, we show an example of using a different initial basis function to construct the accuracy-enhancing filter $K^{(2k+1,k+1)}$.  We present a numerical verification to show the generic accuracy-enhancing filter can also extract the superconvergence order of $2k+1$ from the DG solution.  

In this example, we construct a series basis functions start with a raised cosine filter for the spectral method in~\cite{Vandeven:1991}, that 
\begin{equation}
    \label{eq:raised_cosine}
    \phi^{(1)} = \left\{
            \begin{array}{cc}
                \frac{1}{2}\left(1+\cos(2\pi x)\right)  , & x \in [-\frac{1}{2},\frac{1}{2}], \\
                0, & \text{otherwise}.
            \end{array}\right.
\end{equation}
Then, we construct the rest basis functions recursively by~\eqref{eq:recursion}.  Here, for a better comparison, we define the initial basis function $\phi^{(1)}$ on the interval $\left[-\frac{1}{2},\frac{1}{2}\right]$ same as the first order central B-spline $\psi^{(1)}$ (or $\chi_{[-\frac{1}{2},\frac{1}{2}]}$).  For general situation, it only requires the initial basis function has a compact support.  We give the analytic formula in~\eqref{eq:raised_cosine2} and plots in Figure~\ref{figure:raised_cosine} of the raised cosine basis functions $\phi^{(k+1)}$ with $k=1,2,3$.
\begin{equation}
    \label{eq:raised_cosine2}
  \begin{split}
    \phi^{(2)}(x) & = \left\{\begin{array}{ll}
        \frac{1}{2}(1 + x) -\frac{1}{4\pi}\sin(2\pi x), & x \in [-1,0), \\
        \frac{1}{2}(1 - x) +\frac{1}{4\pi}\sin(2\pi x), & x \in [0, 1], \\
        0 & \text{otherwise};
    \end{array} \right. \\
    \phi^{(3)}(x) & = \left\{\begin{array}{ll}
        \frac{1}{16}(2x+3)^2 - \frac{1}{8\pi^2}\left(1+\cos(2\pi x)\right), & x \in [-\frac{3}{2},-\frac{1}{2}), \\
        \frac{1}{8}(-4x^2 + 3) + \frac{1}{4\pi^2}\left(1+\cos(2\pi x)\right), & x \in [-\frac{1}{2}, \frac{1}{2}), \\
        \frac{1}{16}(2x-3)^2 -  \frac{1}{8\pi^2}\left(1+\cos(2\pi x)\right), & x \in [\frac{1}{2},-\frac{3}{2}], \\
        0 & \text{otherwise};
    \end{array} \right. \\
    \phi^{(4)}(x) & = \left\{\begin{array}{ll}
        \frac{1}{12}(x+2)^3 + \frac{1}{16\pi^3}\left(-2\pi(x+2)+\sin(2\pi x)\right),        & x \in [-2,-1), \\
        \frac{1}{12}(-3x^3-6x^2+4) + \frac{1}{16\pi^3}\left(2\pi(3x+2)-3\sin(2\pi x)\right), & x \in [-1, 0), \\
        \frac{1}{12}( 3x^3-6x^2+4) + \frac{1}{16\pi^3}\left(2\pi(-3x+2)+3\sin(2\pi x)\right), & x \in [ 0, 1), \\
        \frac{1}{12}(2-x)^3 + \frac{1}{16\pi^3}\left(2\pi(x-2)-\sin(2\pi x)\right),        & x \in [ 1, 2], \\
        0 & \text{otherwise}.
    \end{array} \right.
  \end{split}
\end{equation}

\begin{figure}[!ht]
    \centering
    \begin{minipage}[c]{0.32\textwidth}
        \centerline{$\quad k = 1$}
    \end{minipage}
    \begin{minipage}[c]{0.32\textwidth}
        \centerline{$\quad k = 2$}
    \end{minipage}
    \begin{minipage}[c]{0.32\textwidth}
        \centerline{$\quad k = 3$}
    \end{minipage}

    \begin{minipage}[c]{0.32\textwidth}
        \includegraphics[width=1.\textwidth]{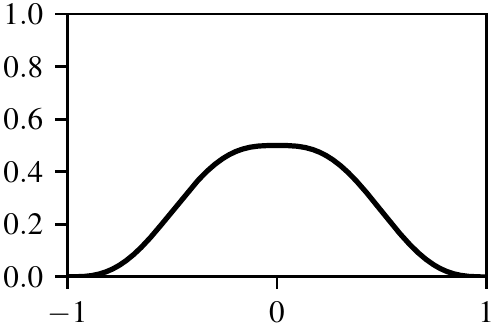}
    \end{minipage}
    \begin{minipage}[c]{0.32\textwidth}
        \includegraphics[width=1.\textwidth]{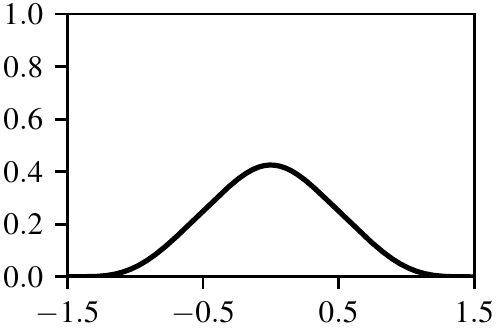}
    \end{minipage}
    \begin{minipage}[c]{0.32\textwidth}
        \includegraphics[width=1.\textwidth]{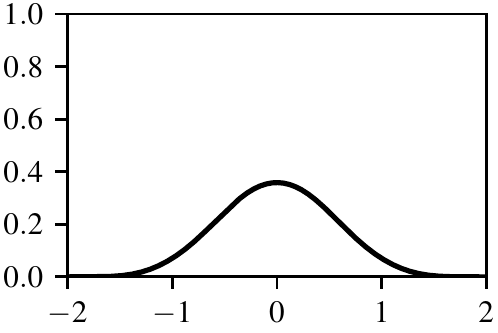}
    \end{minipage}

    \begin{center}
    \caption{\label{figure:raised_cosine}
        The raised cosine basis function $\phi^{(k+1)}$ with $k=1,2,3$ in~\eqref{eq:raised_cosine2}.  The initial basis function $\phi^{(1)}$ chosen as the raised cosine filter~\eqref{eq:raised_cosine}. 
}
     \end{center}
\end{figure}
  
With the raised cosine basis functions, we can compute the linear combination coefficients $c^{(2k+1,k+1)}_\gamma$ by~\eqref{eq:reprod}, and construct the filter $K^{(2k+1,k+1)}$ in~\eqref{eq:general_filter}.  In Figure~\ref{figure:raised_cosine_filter}, we present this new raised cosine filter $K^{(2k+1,k+1)}$ with $k=1, 2, 3$.    Compared to the $k+1$ order central B-spline filter which is a $\mathcal{C}^{k-1}$ function, the $k+1$ order raised cosine filter is a $\mathcal{C}^{k+1}$ function.  We can observe this improvement of smoothness by comparing Figure~\ref{figure-filter} and Figure~\ref{figure:raised_cosine_filter}.  This extra smoothness is especially important when the SIAC filtering is used for visualization, shock regularization, etc.

\begin{rem}
As mentioned earlier, the central B-spline filtering may have the most straightforward formula.  However, the general construction idea of basis functions paves the way for designing new SIAC filters with proven same $2k+1$ superconvergence property and more complicated features.  For example, a DG solution's inter-element discontinuity may hinder its utility in various applications, such as visualization.  For this purpose, one can design a more smooth filter (like the raised cosine filter) to increase the DG solution's inter-element smoothness further.  If needed, one can even create a $\mathcal{C}^\infty_0$ SIAC filter to enhance the smoothness of the filtered solutions, such as choosing the initial basis function as
    \[
        \phi^{(1)} = \left\{
            \begin{array}{cc}
                \exp\left(-\frac{1}{1 - 4x^2}\right), & |x| < \frac{1}{2}, \\
                0, & |x| \geq \frac{1}{2}.
            \end{array}\right.
    \]
Then we can obtain a $\mathcal{C}^{\infty}_0$ filter, which has the same support size as the central B-spline SIAC filter.  It follows that the filtered solution $u^\star_h \in \mathcal{C}^{\infty}$ and also has the accuracy order of $2k+1$. Alternatively, one can use the idea of the Dirac-Delta filter in~\cite{Bohm:2019, Wissink:2018} to design a polynomial based filter with specifically required smoothness property.
\end{rem}

\begin{rem}
We note that for the multi-dimension SIAC filter, it is usually the tensor-product of the one dimension filter, see the two-dimension example in Section~\ref{sec:2d}.  However, an interesting direction of future research is that we can design new multi-dimension SIAC filters beyond the tensor-product assumption.  For example, one can choose a non-separable initial basis function to construct non-separable multi-dimension SIAC filters.  Also, the support of the SIAC can be any shape (instead of a rectangle shape, see Figure~\ref{figure:support_filter2d}) for DG solutions on generic polygon meshes, like the Hexagonal SIAC filter for a hexagonal mesh structure~\cite{Mirzargar:2017}. 
\end{rem}

\begin{figure}[!ht]
    \centering
    \begin{minipage}[c]{0.32\textwidth}
        \centerline{$\quad k = 1$}
    \end{minipage}
    \begin{minipage}[c]{0.32\textwidth}
        \centerline{$\quad k = 2$}
    \end{minipage}
    \begin{minipage}[c]{0.32\textwidth}
        \centerline{$\quad k = 3$}
    \end{minipage}

    \begin{minipage}[c]{0.32\textwidth}
        \includegraphics[width=1.\textwidth]{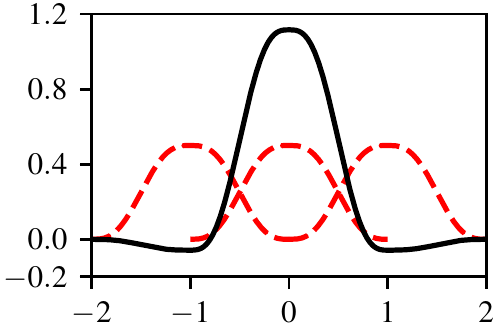}
    \end{minipage}
    \begin{minipage}[c]{0.32\textwidth}
        \includegraphics[width=1.\textwidth]{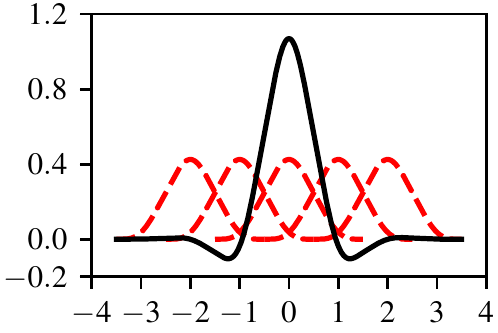}
    \end{minipage}
    \begin{minipage}[c]{0.32\textwidth}
        \includegraphics[width=1.\textwidth]{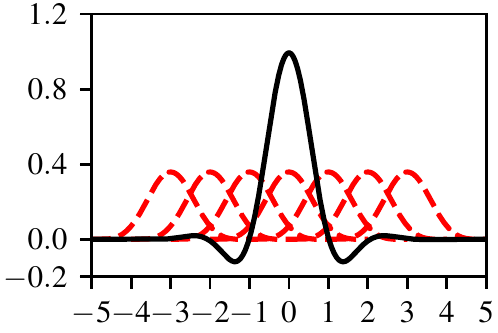}
    \end{minipage}

    \begin{center}
    \caption{\label{figure:raised_cosine_filter}
        The raised cosine filter $K^{(2k+1,k+1)}$ constructed by a raised cosine basis function $\phi^{(k+1)}$ in~\eqref{eq:raised_cosine2}, with $k=1, 2, 3$. 
}
     \end{center}
\end{figure}

\subsubsection{Numerical Verification}
In this part, we illustrate the capabilities of the generic accuracy-enhancing filtering and compare it to the central B-spline SIAC filtering through a numerical example.  Consider the one-dimensional advection equation
\begin{equation}
    \label{eq:advection}
    \begin{split}
        u_t + u_x & = 0, \quad (x,t) \in [0,1]\times (0, T]\\
        u(x,0) & =  \sin(2\pi x)
    \end{split}
\end{equation}
with the final time $T = 1$ and periodic boundary conditions over uniform meshes.  The $L^2$ norm errors and respective accuracy orders are given in Table~\ref{table:general_filter}.  We can see that the DG solution has the regular accuracy order of $k+1$.  For the filtered solutions, we observe that both the central B-spline filtering and the newly constructed raised cosine filtering have the superconvergence order of $2k+1$.  Compared to the original SIAC filtering, the raised cosine filtering leads to a slightly smaller error in the $L^2$ norm.  In Figure~\ref{figure:general_filter}, we show the point-wise error plots before and after applying the central B-spline filtering and the raised cosine filtering.  We note that both filters recover smoothness in the DG solution as well as reduce the error.   

\begin{table}\centering 
\scalebox{1.}{
\begin{tabular}{ccccccccccc}
\toprule
  & & & \multicolumn{2}{c}{DG} &&  \multicolumn{2}{c}{Central B-spline} && \multicolumn{2}{c}{Raised Cosine} \\ 
\cmidrule{4-5} \cmidrule{7-8} \cmidrule{10-11} 
Degree & Elements & & Error & Order && Error & Order && Error & Order \\ 
\midrule
        & 20 && 4.60e-03 &   --  &&  1.97e-03 &   --  && 1.95E-03 &   --\\ 
$k = 1$ & 40 && 1.09e-03 &  2.08 &&  2.44e-04 &  3.02 && 2.42E-04 &  3.01 \\ 
        & 80 && 2.67e-04 &  2.02 &&  3.02e-05 &  3.01 && 3.02E-05 &  3.01 \\ 
        &160 && 6.65e-05 &  2.01 &&  3.76e-06 &  3.01 && 3.76E-06 &  3.00 \\ 
\midrule
        & 20 && 1.07e-04 &   --  &&  4.10e-06 &   --  && 3.42E-06 &   --\\ 
$k = 2$ & 40 && 1.34e-05 &  3.00 &&  9.42e-08 &  5.44 && 8.41E-08 &  5.35\\ 
        & 80 && 1.67e-06 &  3.00 &&  2.40e-09 &  5.30 && 2.32E-09 &  5.18 \\ 
        &160 && 2.09e-07 &  3.00 &&  6.63e-11 &  5.18 && 7.49E-11 &  4.95 \\ 
\midrule
        & 20 && 2.06e-06 &   --  &&  6.97e-08 &   --  && 5.09E-08 &   -- \\ 
$k = 3$ & 40 && 1.29e-07 &  4.00 &&  2.83e-10 &  7.95 && 2.07E-10 &  7.94\\ 
        & 80 && 8.07e-09 &  4.00 &&  1.14e-12 &  7.95 && 8.40E-13 &  7.94\\ 
        &160 && 5.04e-10 &  4.00 &&  4.67e-15 & 7.93 && 3.60E-15 &  7.87\\ 
\bottomrule
\end{tabular}}
\caption{Convergence tests for advection equation~\eqref{eq:advection} for the DG method with the filtering techniques.  Here, the finial time $T = 1$.  The $L^2$ norm error of the DG solution has the accuracy order of $k+1$, and the solutions after filtering with both central B-spine filtering and the raised cosine filtering have the same superconvergence order of $2k+1$.}
\label{table:general_filter}
\end{table}

\begin{figure}[!ht]
    \centering
    \begin{minipage}[c]{0.32\textwidth}
        \centerline{DG}
    \end{minipage}
    \begin{minipage}[c]{0.32\textwidth}
        \centerline{$\quad$Central B-spline}
    \end{minipage}
    \begin{minipage}[c]{0.32\textwidth}
        \centerline{$\qquad$Raised Cosine}
    \end{minipage}

    \begin{minipage}[c]{0.32\textwidth}
        \includegraphics[width=1.\textwidth]{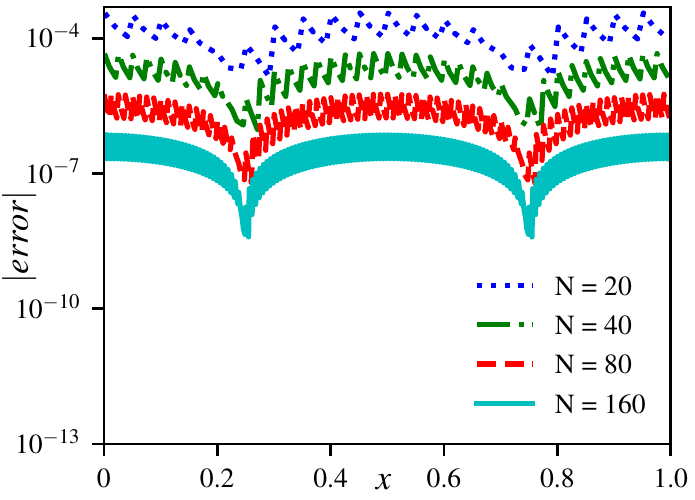}
    \end{minipage}
    \begin{minipage}[c]{0.32\textwidth}
        \includegraphics[width=1.\textwidth]{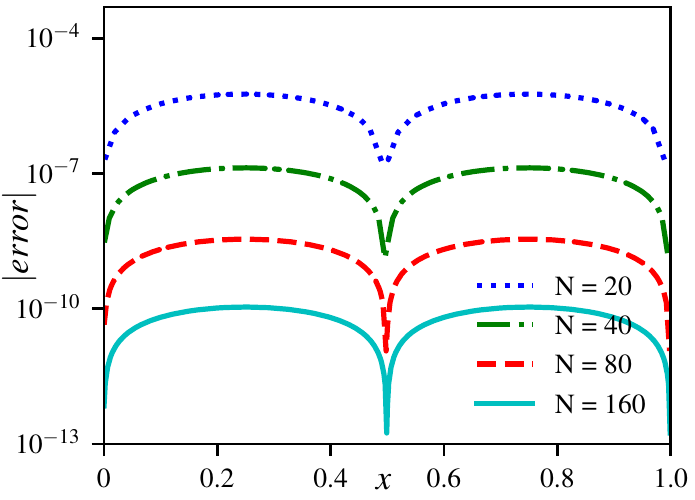}
    \end{minipage}
    \begin{minipage}[c]{0.32\textwidth}
        \includegraphics[width=1.\textwidth]{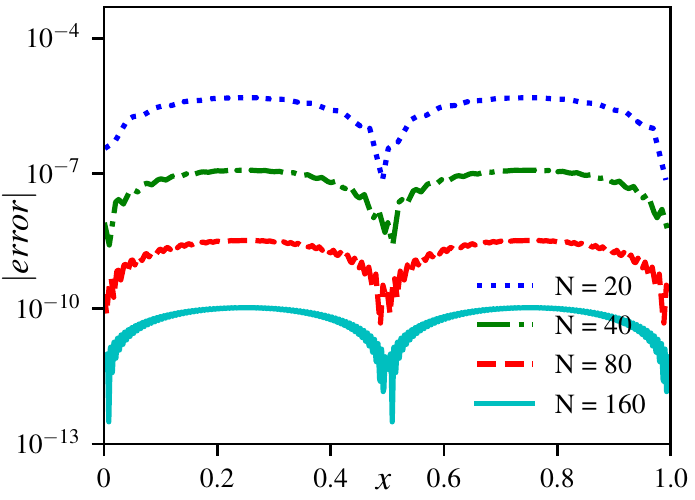}
    \end{minipage}
    \begin{center}
    \caption{\label{figure:general_filter}
        The point-wise error plots for advection equation~\eqref{eq:advection} for the DG method with the filtering techniques with polynomial $\mathbb{P}^2$.  Here, the final time $T = 1$.  We observe that the compared to the DG solution, the solutions after filtering with both central B-spine filtering and the raised cosine filtering have recovered smoothness in the approximation and reduces the error.
}
     \end{center}
\end{figure}

\section{The Distribution of the Basis Functions}
\label{sec:compact}
During the previous discussions, we addressed how to construct the general basis functions for the SIAC filter while still preserving the same accuracy-enhancing ability.  In this section, we treat another often overlooked component when constructing an accuracy-enhancing filter: the basis functions' distribution.  For convenience, in this section, we mainly use the central B-spline as the basis function.  However,  one can naturally extend the conclusion to other basis functions, such as the raised cosine basis functions in~\eqref{eq:raised_cosine2}.        

Recall the formula of the original SIAC filter,
\begin{equation}
    \label{eq:filter1}
    K^{(2k+1,k+1)}(x) = \sum\limits_{\gamma=0}^{2k} c_\gamma^{(2k+1,k+1)}\psi^{(k+1)}(x - x_\gamma),
\end{equation}
where
\[
    x_\gamma = -k+\gamma
\]
presents the central point of the basis function $\psi^{(k+1)}(x-x_\gamma)$.  That is, the basis functions $\left\{\psi^{(k+1)(x-x_\gamma}\right\}_\gamma$ are uniformly distributed with distance $1$.  It leads the support size of the filter $K^{(2k+1,k+1)}$ to be $3k+1$, while the scaled filter $K_H^{(2k+1,k+1)}$ has a support size of $(3k+1)H$.  For the uniform meshes, when we take the scaling $H=h$, the filtering convolution will involve $3k+1$ or $3k+2$ elements during computing. In Table~\ref{tab:filter_support},  we present the relation of the polynomial degree $k$ of the DG approximation, the number of B-splines, and the involved DG elements during the filtering convolution.  Although the filter has compact support, the support size is still quite large, especially for the higher-order case and the multi-dimension case (tensor product of the one-dimensional filter).   A large support size may be problematic when introducing boundary conditions, when the mesh is highly unstructured or when the solution is lack of smoothness, etc.  Also, the support size of the filter is an important factor for the computational cost.  
\begin{table}[!ht]
\begin{center}
\begin{tabular}{ccccc}
\toprule
$k$ & Number B-Splines & $H$ & Number Elements  & Expected Accuracy\\ 
\midrule
1 & 3 & $h$    & 4 $\sim$ 5 & 3 \\
  &   & $0.5h$ & 2 $\sim$ 3 & 2 \\
 \midrule
2 & 5 & $h$    & 7 $\sim$ 8 & 5 \\
  &   & $0.5h$ & 4 $\sim$ 5 & 3 \\
  \midrule
3 & 7 & $h$    & 10 $\sim$ 11 & 7 \\
  &   & $0.5h$ & 5  $\sim$  6 & 4\\
 \bottomrule
\end{tabular}
\caption{\label{tab:filter_support} The choice of typical filter parameters (order, number of B-splines, and the scaling) affect the number of elements for the filtering convolution and possible order of accuracy.  Assume the (uniform) mesh size is $h$.  It shows the filter scaling $H$ equal to the $h$ and $0.5h$.  For a scaling small than $h$ ($H < h)$ only $\mathcal{O}(h^{k+1})$ can be expected.}
\end{center}
\end{table}

It is obvious that one can benefit from reducing the filter's support.  One direct way to do this is to let the filter scaling $H$ smaller than $h$ (for example, see $H = 0.5h$ in Table~\ref{tab:filter_support}).  However, choosing $H < h$ will damage the superconvergence property, and only the regular accuracy order of $k+1$ can be expected in the $L^2$ norm, see~\cite{King:2012}.  Due to the loss of the superconvergence property, the filtered solution usually has worse accuracy than using the scaling $H = h$.

To reduce the filter's support size while still preserves the superconvergence property, in this paper, we propose to change the distribution of the basis functions, namely the $x_\gamma$ given in~\eqref{eq:cpoints}.  In literature, there are many generalizations of the SIAC filtering.  However, the distribution of the basis functions is always an overlooked factor.  The basis functions' distribution has never gone beyond a uniform distribution with distance $1$ since the beginning of this filtering technique.  In the rest of this section, we will discuss using the different distribution for the basis functions and propose a compact SIAC filter that significantly reduces the original SIAC filter's support size.

\subsection{The Compact SIAC Filtering}
We note that the support size ($3k+1$) of the SIAC filter comes from two parts.  The first part is the support size ($k+1$) of the basis function used to construct the filter.  As discussed earlier, it can not be reduced without damaging the accuracy-enhancing ability of the filtering.  The second part is the distribution of the basis functions.  
As mentioned earlier, the original SIAC filter samples the central points $x_\gamma$ of its basis functions as follows
\[
    x_\gamma \in \left\{-k, -(k-1), \ldots, k-1, k\right\}.
\]
This distribution contributes a size of $2k$ out of the total support size of $3k+1$.  For this part, we can sample the basis functions more closely to reduce the support size of the filter.   That is, instead of using $x_\gamma$ given in~\eqref{eq:cpoints}, we compress it with a compress parameter $\epsilon$ ($0 < \epsilon < 1$), 
\begin{equation}
    \label{eq:compact_nodes}
    \epsilon x_\gamma = \epsilon(-k + \gamma).
\end{equation} 
In this way, we can introduce the following compact SIAC filter
\begin{equation}
    \label{eq:compact_filter}
    K^{(2k+1,k+1)}_\epsilon(x) = \sum\limits_{\gamma=0}^{2k} c_\gamma^{(2k+1,k+1)}\psi^{(k+1)}(x - \epsilon x_\gamma).
\end{equation}
The support size of the compact SIAC filter~\eqref{eq:compact_filter} is $(2\epsilon+1)k+1$, which is smaller than the support size of the original SIAC filter ($3k+1$).  For example, we present the compact SIAC filters with $\epsilon = 0.5$ in Figure~\ref{figure:compact_filter} (see the original SIAC filter in Figure~\ref{figure-filter}).  In addition, we denote the scaled compact filter as 
\[
    K^{(2k+1,k+1)}_{\epsilon,H} = \frac{1}{H}K^{(2k+1,k+1)}_\epsilon\left(\frac{x}{H}\right).
\]

\begin{figure}[!ht]
    \centering

    \begin{minipage}[c]{0.32\textwidth}
        \centerline{$\quad k = 1$}
        \includegraphics[width=1.\textwidth]{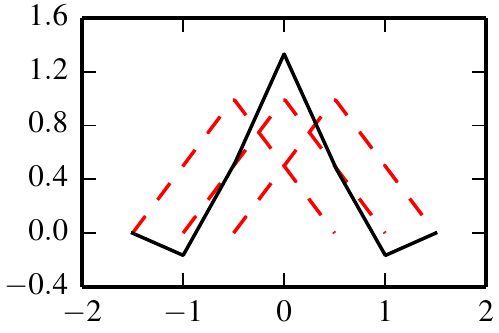}
    \end{minipage}
    \begin{minipage}[c]{0.32\textwidth}
        \centerline{$\quad k = 2$}
        \includegraphics[width=1.\textwidth]{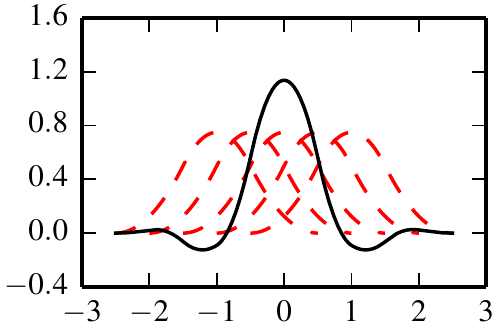}
    \end{minipage}
    \begin{minipage}[c]{0.32\textwidth}
        \centerline{$\quad k = 3$}
        \includegraphics[width=1.\textwidth]{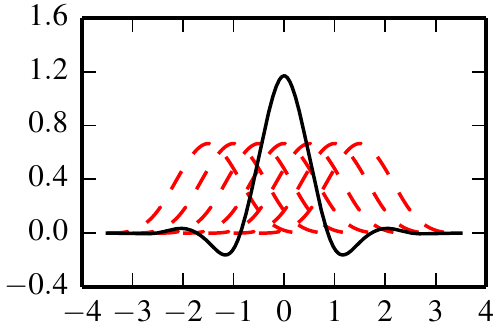}
    \end{minipage}

    \begin{center}
    \caption{\label{figure:compact_filter}
        Solid black lines represent the compact SIAC filter $K^{(2k+1,k+1)}_\epsilon$ given in~\eqref{eq:compact_filter} with $k = 1, 2, 3$, dashed red lines represent the respect central B-splines.  The compress parameter $\epsilon = 0.5$.
}
     \end{center}
\end{figure}

In Table~\ref{tab:compact_filter_support},  we present the relation of the polynomial degree $k$ of the DG approximation, the number of B-splines, the compress parameter $\epsilon$ and the involved DG elements during filtering convolution for the compact SIAC filer (with scaling $H = h$). 

\begin{table}[!ht]
\begin{center}
\begin{tabular}{ccccc}
\toprule
$k$ & Number B-Splines & $\epsilon$ & Number Elements  & Expected Accuracy\\ 
\midrule
1 & 3 & $1/2$    & 3 $\sim$ 4 & 3 \\
 \midrule
2 & 5 & $1/2$    & 5 $\sim$ 6 & 5 \\
  &   & $1/4$    & 4 $\sim$ 5 & 5 \\
  \midrule
3 & 7 & $1/2$    & 7 $\sim$ 8 & 7 \\
  &   & $1/6$    & 5 $\sim$ 6 & 7\\
 \bottomrule
\end{tabular}
\caption{\label{tab:compact_filter_support} The choice of typical filter parameters (order, number of B-splines, and compress parameter) affects the number of elements for the filtering convolution and possible order of accuracy.  Assume the (uniform) mesh size is $h$.   The filter scaling equals to the (uniform) mesh size ($H=h)$.  Here, the compress parameter $\epsilon = 0.5, \frac{1}{2k}$.}
\end{center}
\end{table}

Here, we discuss options for choosing the compress parameter, $\epsilon$,  in~\eqref{eq:compact_nodes}, which is sensitive to scale.  If $\epsilon$ is close to $1$, the support size is still quite large.  However, if $\epsilon$ is too small, the linear system~\eqref{eq:reprod} used to solve the coefficients $c_\gamma^{(2k+1,k+1)}$ of the filter will close to singular.  Then, the coefficients $c_\gamma^{(2k+1,k+1)}$ are contaminated by floating-point round-off error.  The best $\epsilon$ to use for a compact filter can be optimized to balance these two quantifiable trade-offs.  In this paper, after some numerical testing, we choose the $\epsilon = \frac{1}{2k}$.  Then, the support size becomes $k+2$, which is slightly larger than the optimal size of $k+1$ while avoiding being affected too much by the round-off problem.  That is, we limit the central point of the basis functions into interval $\left[-\frac{1}{2}, \frac{1}{2}\right]$,  
\[
    \epsilon x_\gamma \in \left\{-\frac{1}{2}, \ldots, \frac{1}{2}\right\}.
\]
Also, we note that the compact filter's support is independent of the number of basis functions.    
 
We note the way of the distribution of the basis functions has no effect on Property~\ref{prop-divided} and Property~\ref{prop:reprod}, the essential component to prove the superconvergence property.  For the filtered solutions, we can easily prove that the superconvergence of order $2k+1$ can be achieved.
\begin{thm}
    \label{thm:compact}
    Under the same conditions in Theorem~\ref{thm-1}, the compact filter defined in~\eqref{eq:compact_filter} with $\epsilon > 0$, then
    \[ 
        \|u - K_{\epsilon, h}^{(2k+1,k)}\star u_h\|_{0,\Omega_0} \leq C h^{2k+1}.
    \]
\end{thm}
\begin{proof}
    The proof is analogous to Theorem~\ref{thm:general}.
\end{proof}

\begin{rem}
In this section, we only present the discussion on uniformly distribute $x_\gamma$ (or say $\epsilon x_\gamma$).  From the theoretical point of view (Theorem~\ref{thm:compact}), there is no need to limit $x_\gamma$ to be a uniform distribution.  Besides the uniform distribution, one can choose $x_\gamma$ as the Gauss points, Chebyshev points, etc.  However, we note that there is no noticeable benefits to use non-uniform distributions respect to superconvergence extracting, the support size, the accuracy level of the filtered solutions.  Therefore, we skip the discussion of using those non-uniform distributions.   
\end{rem}

\subsection{Numerical Results}
In this section, we provide numerical results to show the superconvergence of the compact SIAC filter in practice and compare it with the original SIAC filtering.  For all the numerical examples, we choose the compress parameter $\epsilon = \frac{1}{2k}$.

\subsubsection{Symmetric Filtering}
For the first set of examples, we consider the same setting of the advection equation~\eqref{eq:advection} presented in the previous section.

The $L^2$ norm errors and respective accuracy orders are given in Table~\ref{table:compact_filtering}.  Compared to the DG solutions with the regular accuracy order of $k+1$, the filtered solutions (with both the original and compact SIAC filtering) have the superconvergence order of $2k+1$.  Furthermore, we observe that compared to the original SIAC filter, the compact filter has a much smaller support size while shows noticeable better accuracy for higher-order cases.  For example, in the $k=3$ case, the compact SIAC filter has only half the original SIAC filter's support size, while the filtered errors are more than $10$ times smaller.  In Figure~\ref{figure:compact_filtering}, we show the point-wise error plots before and after applying the original and compact SIAC filtering.  We note that both filters recover smoothness in the DG solution as well as reduce the error.   

\begin{table}\centering 
\scalebox{1.}{
\begin{tabular}{ccccccccccc}
\toprule
  & & & \multicolumn{2}{c}{DG} &&  \multicolumn{2}{c}{Original SIAC} && \multicolumn{2}{c}{Compact SIAC} \\ 
\cmidrule{4-5} \cmidrule{7-8} \cmidrule{10-11} 
Degree & Elements & & Error & Order && Error & Order && Error & Order \\ 
\midrule
        & 20 && 4.60e-03 &   --  &&  1.97e-03 &   --  && 1.94E-03 &   --\\ 
$k = 1$ & 40 && 1.09e-03 &  2.08 &&  2.44e-04 &  3.02 && 2.42E-04 &  3.00 \\ 
        & 80 && 2.67e-04 &  2.02 &&  3.02e-05 &  3.01 && 3.01E-05 &  3.00 \\ 
        &160 && 6.65e-05 &  2.01 &&  3.76e-06 &  3.01 && 3.76E-06 &  3.00 \\ 
\midrule
        & 20 && 1.07e-04 &   --  &&  4.10e-06 &   --  && 2.27E-06 &   --\\ 
$k = 2$ & 40 && 1.34e-05 &  3.00 &&  9.42e-08 &  5.44 && 6.57E-08 &  5.11\\ 
        & 80 && 1.67e-06 &  3.00 &&  2.40e-09 &  5.30 && 2.03E-09 &  5.02 \\ 
        &160 && 2.09e-07 &  3.00 &&  6.63e-11 &  5.18 && 7.03E-11 &  4.85 \\ 
\midrule
        & 20 && 2.06e-06 &   --  &&  6.97e-08 &   --  && 5.26E-09 &   -- \\ 
$k = 3$ & 40 && 1.29e-07 &  4.00 &&  2.83e-10 &  7.95 && 2.44E-11 &  7.75\\ 
        & 80 && 8.07e-09 &  4.00 &&  1.14e-12 &  7.95 && 1.25E-13 &  7.61\\ 
        &160 && 5.04e-10 &  4.00 &&  4.67e-15 & 7.93 && 8.01E-16 &  7.29\\ 
\bottomrule
\end{tabular}}
\caption{Convergence tests for advection equation~\eqref{eq:advection} for the DG method with the filtering techniques.  Here, the final time $T = 1$.  We observe that the $L^2$ norm error of the DG solution has the accuracy order of $k+1$, and the solutions after filtering with the regular SIAC  filtering and the compact SIAC filtering have the same superconvergence order of $2k+1$.}
\label{table:compact_filtering}
\end{table}

\begin{figure}[!ht]
    \centering
    \begin{minipage}[c]{0.32\textwidth}
        \centerline{DG}
    \end{minipage}
    \begin{minipage}[c]{0.32\textwidth}
        \centerline{$\quad$Original SIAC}
    \end{minipage}
    \begin{minipage}[c]{0.32\textwidth}
        \centerline{$\qquad$Compact SIAC}
    \end{minipage}

    \begin{minipage}[c]{0.32\textwidth}
        \includegraphics[width=1.\textwidth]{figures/linear_dg_p2.pdf}
    \end{minipage}
    \begin{minipage}[c]{0.32\textwidth}
        \includegraphics[width=1.\textwidth]{figures/linear_siac_p2.pdf}
    \end{minipage}
    \begin{minipage}[c]{0.32\textwidth}
        \includegraphics[width=1.\textwidth]{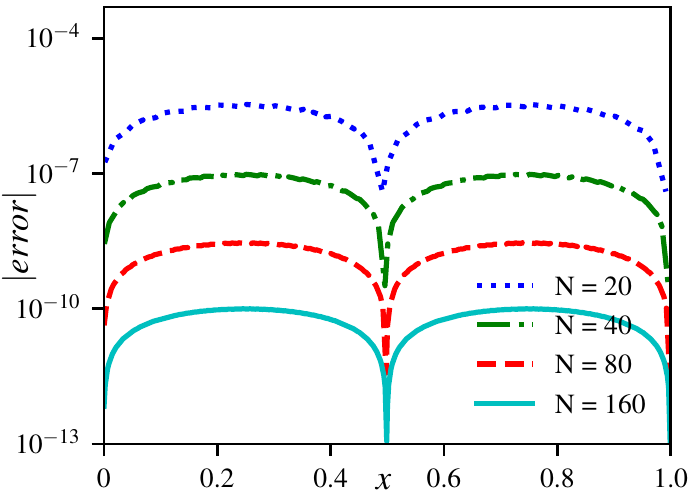}
    \end{minipage}
    \begin{center}
    \caption{\label{figure:compact_filtering}
        The point-wise error plots for advection equation~\eqref{eq:advection} for the DG method with the filtering techniques with polynomial $\mathbb{P}^2$.  Here, the final time $T = 1$.  We observe that the compared to the DG solution, the solutions after filtering with both the original and compact SIAC filtering have recovered the smoothness in the approximation and reduced the error.
}
     \end{center}
\end{figure}

\subsubsection{Non-symmetric Filtering} 
We can see the compact filter works at least as good as the original filter, the same accuracy order, the same smoothness, and better accuracy through the previous example.  Since the compact filter has significantly reduced support size, it can also help solve many issues caused by the filter's large support size, such as dealing with boundary problems.

Throughout the literature of the SIAC filtering, there are always some difficulties in dealing with non-periodic boundaries.  Due to the symmetric feature of the SIAC filter, one can not apply it directly near the boundary regions without periodic boundary conditions.   To overcome this issue, near the boundary regions, one has to shift the SIAC filter according to the position of point needs to be filtered -- the position-dependent SIAC filters~\cite{Li:2015, vanSlingerland:2011}.  However, the position-dependent filters usually have worse numerical performance than the symmetric SIAC filter.  Therefore, one should try not to use them unless necessary.   The compact filter with a smaller support size will be more suitable than the original SIAC filter in this situation.  

In the second numerical example, we consider the same advection equation~\eqref{eq:advection}.  Unlike the previous example, we have used the symmetric SIAC filtering with periodic boundary conditions.  In this example, we use the position-dependent SIAC filtering to deal with the boundary regions and use the symmetric SIAC filtering for the interior regions.  The $L^2$ norm errors and respective accuracy orders are given in Table~\ref{table:onesided_filtering}.  Like the previous example, both the original SIAC and compact filtered solutions have a superconvergence order of $2k+1$.  However, we note that the compact SIAC filtering provides much better accuracy for the $L^2$ norm error.  Also, we show the point-wise error plots before and after applying the original and compact SIAC filtering in Figure~\ref{figure:onesided_filtering}.  From Figure~\ref{figure:onesided_filtering}, We can see that when using the compact filter, the boundary regions that need use the position-dependent filter are much smaller than the original filter case ($\frac{k+2}{2}$ elements vs. $\frac{3k+1}{2}$ elements).

\begin{table}\centering 
\scalebox{1.}{
\begin{tabular}{ccccccccccc}
\toprule
  & & & \multicolumn{2}{c}{DG} &&  \multicolumn{2}{c}{original SIAC} && \multicolumn{2}{c}{compact SIAC} \\ 
\cmidrule{4-5} \cmidrule{7-8} \cmidrule{10-11} 
Degree & Elements & & Error & Order && Error & Order && Error & Order \\ 
\midrule
        & 20 && 4.60e-03 &   --  &&  3.67E-03 &   --  && 2.33E-03 &   --\\ 
$k = 1$ & 40 && 1.09e-03 &  2.08 &&  3.95E-04 &  3.21 && 2.72E-04 &  3.10 \\ 
        & 80 && 2.67e-04 &  2.02 &&  4.15E-05 &  3.25 && 3.22E-05 &  3.08 \\ 
        &160 && 6.65e-05 &  2.01 &&  4.54E-06 &  3.19 && 3.89E-06 &  3.05 \\ 
\midrule
        & 20 && 1.07e-04 &   --  &&  4.84E-04 &   --  && 2.64E-05 &   --\\ 
$k = 2$ & 40 && 1.34e-05 &  3.00 &&  1.69E-05 &  4.84 && 6.85E-07 &  5.27\\ 
        & 80 && 1.67e-06 &  3.00 &&  4.10E-07 &  5.36 && 1.58E-08 &  5.44 \\ 
        &160 && 2.09e-07 &  3.00 &&  9.28E-09 &  5.47 && 3.56E-10 &  5.47 \\ 
\midrule
        & 20 && 2.06e-06 &   --  &&  3.46E-05 &   --  && 3.57E-07 &   -- \\ 
$k = 3$ & 40 && 1.29e-07 &  4.00 &&  9.42E-07 &  5.20 && 2.59E-09 &  7.11\\ 
        & 80 && 8.07e-09 &  4.00 &&  6.47E-09 &  7.19 && 1.52E-11 &  7.41\\ 
        &160 && 5.04e-10 &  4.00 &&  3.76E-11 &  7.43 && 8.55E-14 &  7.48\\ 
\bottomrule
\end{tabular}}
\caption{Convergence tests for advection equation~\eqref{eq:advection} for the DG method with the position-dependent filtering techniques.  Here, the final time $T = 1$.  We observe that the $L^2$ norm error of the DG solution has the accuracy order of $k+1$, and the solutions after filtering with the original SIAC filtering and the compact SIAC filtering have the same superconvergence order of $2k+1$.  However, we note that the compact SIAC filtering provides much better accuracy for the error.}
\label{table:onesided_filtering}
\end{table}

\begin{figure}[!ht]
    \centering
    \begin{minipage}[c]{0.32\textwidth}
        \centerline{DG}
    \end{minipage}
    \begin{minipage}[c]{0.32\textwidth}
        \centerline{$\quad$Original SIAC}
    \end{minipage}
    \begin{minipage}[c]{0.32\textwidth}
        \centerline{$\qquad$Compact SIAC}
    \end{minipage}

    \begin{minipage}[c]{0.32\textwidth}
        \includegraphics[width=1.\textwidth]{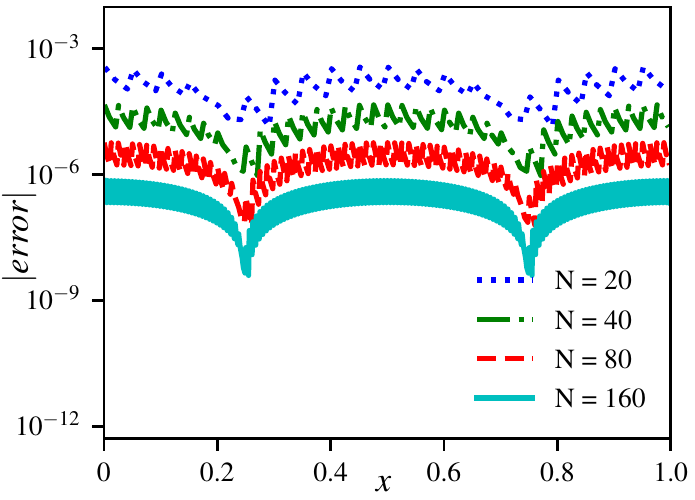}
    \end{minipage}
    \begin{minipage}[c]{0.32\textwidth}
        \includegraphics[width=1.\textwidth]{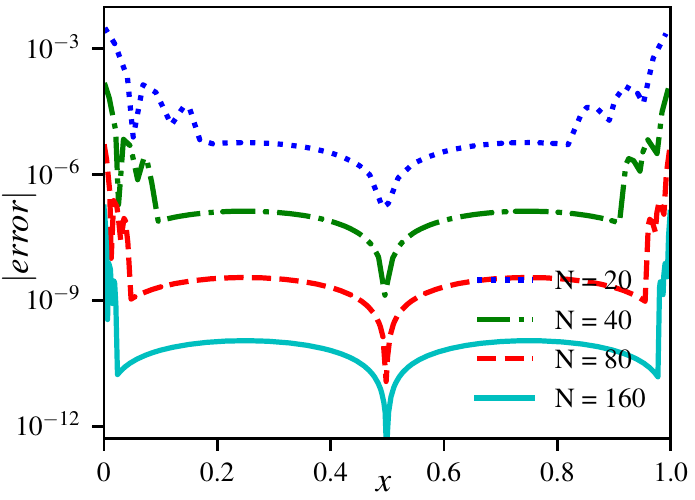}
    \end{minipage}
    \begin{minipage}[c]{0.32\textwidth}
        \includegraphics[width=1.\textwidth]{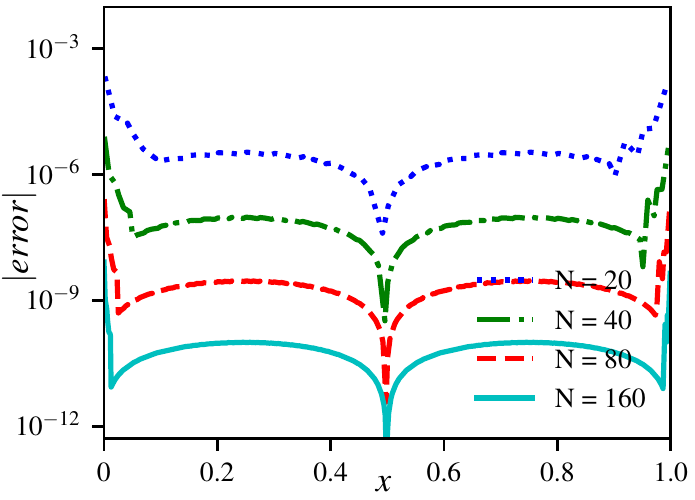}
    \end{minipage}
    
    \begin{minipage}[c]{0.32\textwidth}
        \includegraphics[width=1.\textwidth]{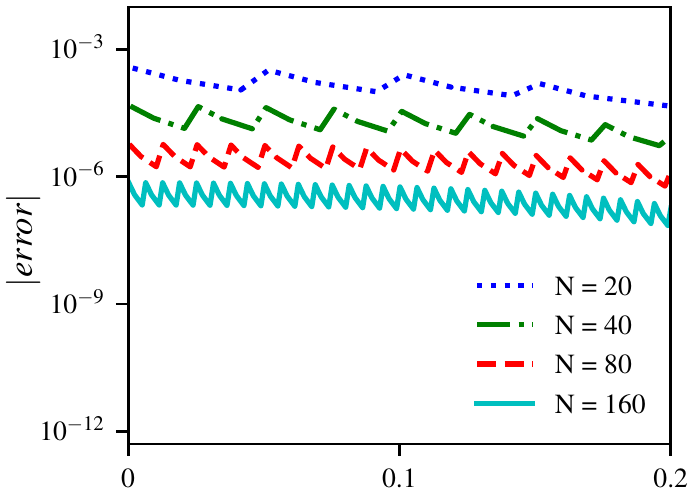}
    \end{minipage}
    \begin{minipage}[c]{0.32\textwidth}
        \includegraphics[width=1.\textwidth]{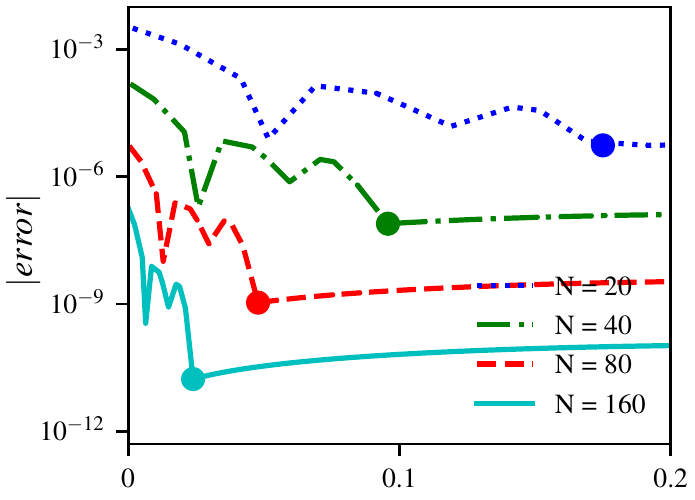}
    \end{minipage}
    \begin{minipage}[c]{0.32\textwidth}
        \includegraphics[width=1.\textwidth]{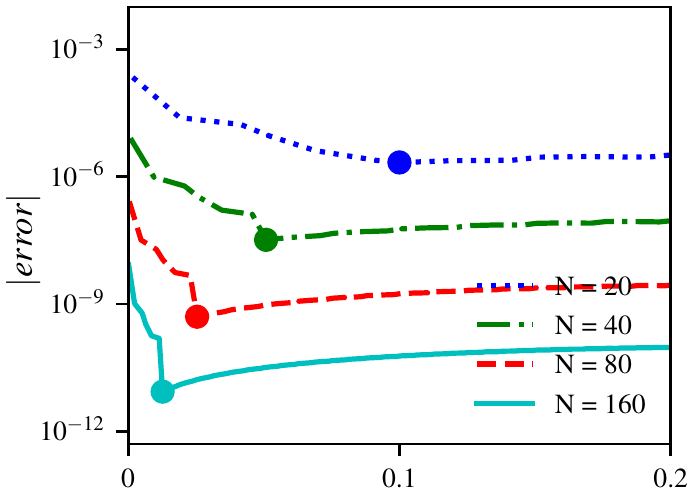}
    \end{minipage}
    \begin{center}
    \caption{\label{figure:onesided_filtering}
        The point-wise error plots for advection equation~\eqref{eq:advection} for the DG method with the position-dependent filtering with polynomial $\mathbb{P}^2$.  
        The bottom row presents zoomed plots for the left boundary region, and we mark the interface for the position-dependent filtering and the symmetric filtering by a big dot.  When using the compact filter, the boundary regions need to use the position-dependent filtering are smaller than using the original filter.
}
     \end{center}
\end{figure}

\subsubsection{The 2D Filtering}
\label{sec:2d}
At last, we show an application of the filter for a two-dimensional advection equation:
\begin{equation}
\label{eq:advection2d}
    \begin{split}
        u_t + u_x + u_y&= 0, \quad (x, y)\in[0,2\pi]\times[0, 2\pi],\ t \geq 0 \\
        u(x, y, 0) &= \sin(x + y).
    \end{split}
\end{equation}
We consider uniform rectangular meshes and apply the 2D SIAC filters (tensor-product of 1D filter) for the DG solutions.  First, we compare the support size of the original SIAC filter and the compact SIAC filter in Figure~\ref{figure:support_filter2d}.  One can observe that the difference of the support size is obvious, especially for the higher-order case (see, $k=2,3$).  In Table~\ref{table:filtered2d}, we compare the $L^2$ norm errors and convergence rates before and after applying the SIAC filter to the solution. We can see that the filter raises the order from $k+1$ to $2k+1$.  Further, the accuracy of the solution is largely improved. This is emphasized in Figure~\ref{fig:filtered2d}, where we provide the contour error plots for the particular case of $k=3$ and $80\times 80$ elements. Notice how the filter also recovers smoothness in error.

\begin{figure}[!ht]
    \centering
    \begin{minipage}[c]{0.48\textwidth}
        \centerline{$\quad$Original SIAC}
    \end{minipage}
    \begin{minipage}[c]{0.48\textwidth}
        \centerline{$\qquad$Compact SIAC}
    \end{minipage}

    \begin{minipage}[c]{0.48\textwidth}
        \centering
        \includegraphics[width=0.8\textwidth]{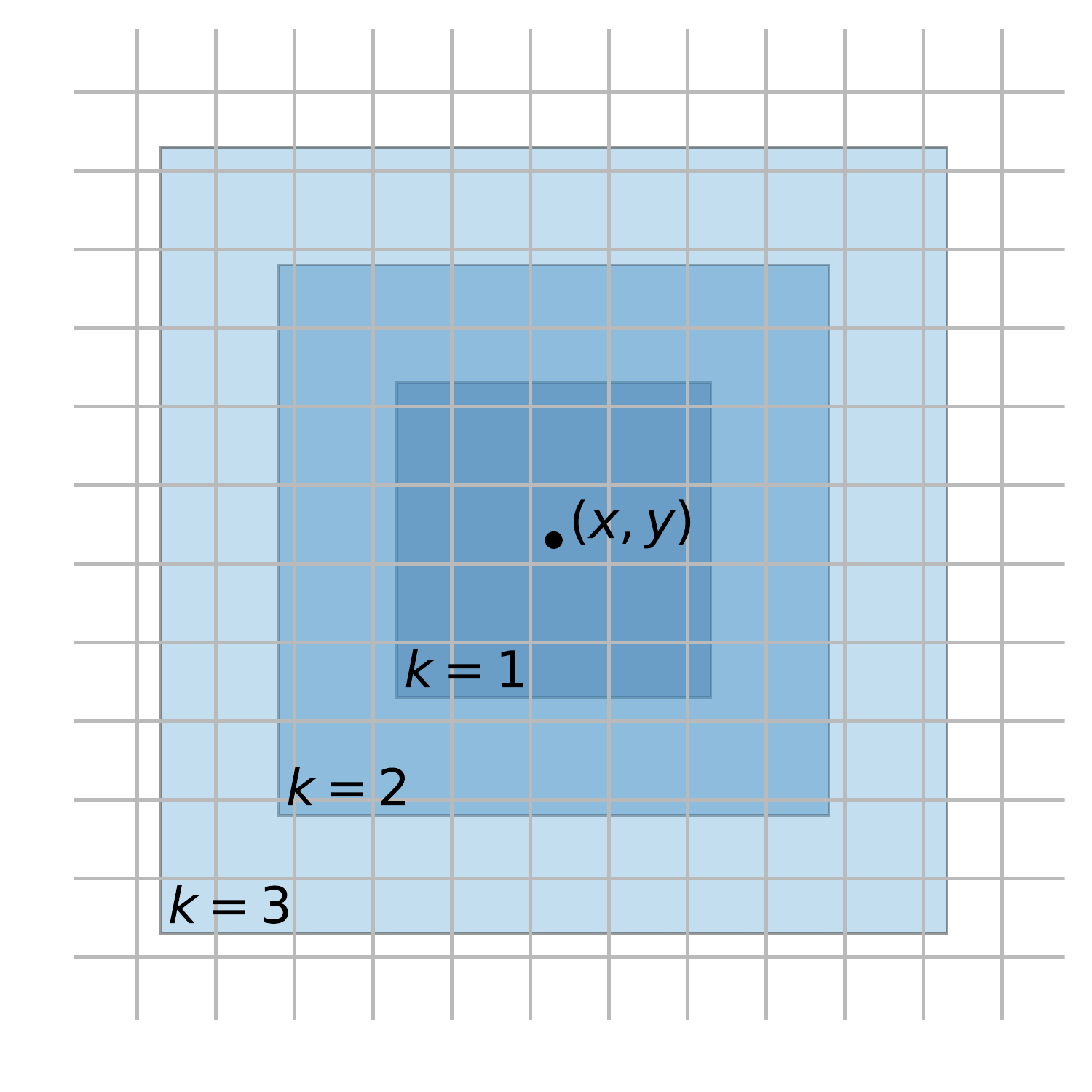}
    \end{minipage}
    \begin{minipage}[c]{0.48\textwidth}
        \centering
        \includegraphics[width=0.8\textwidth]{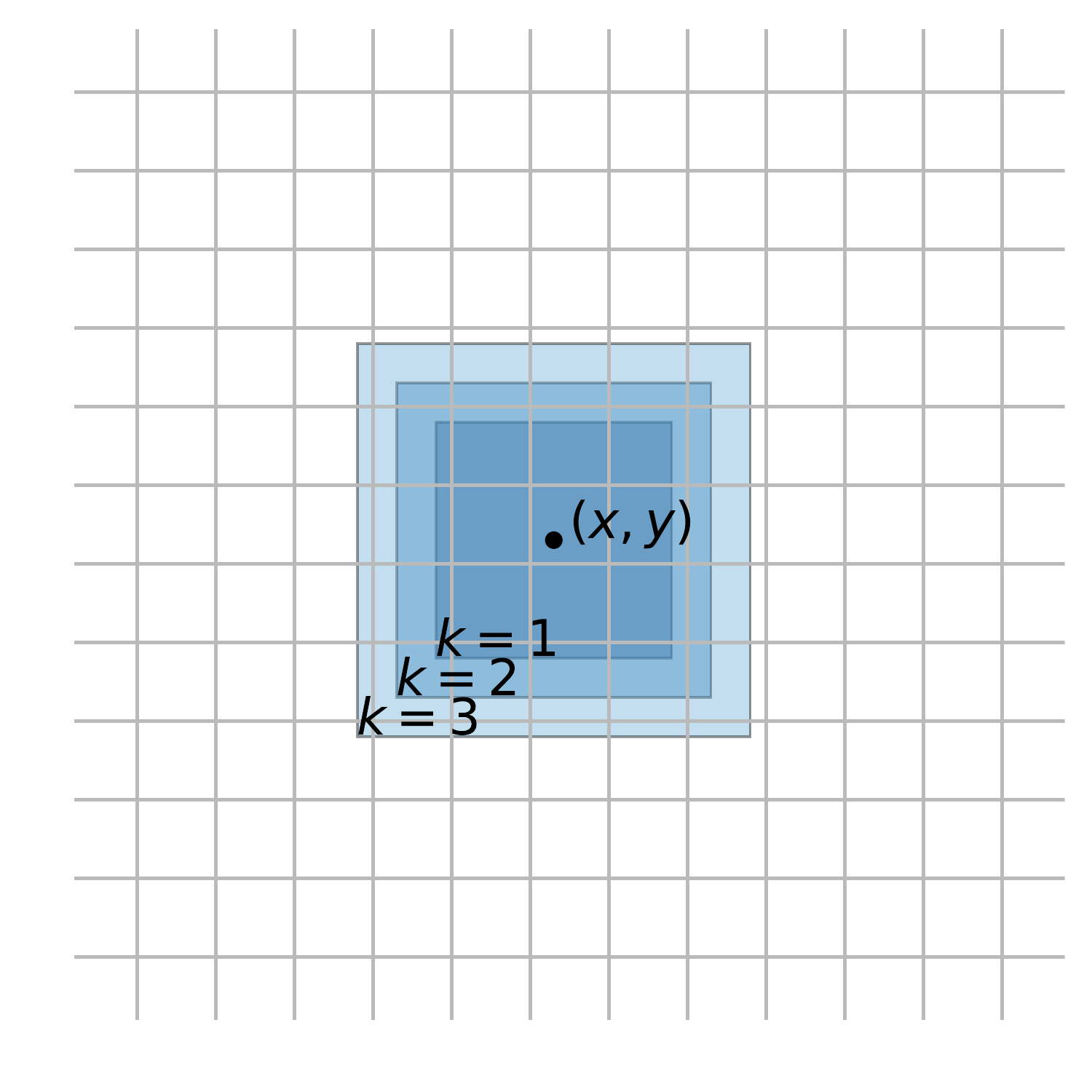}
    \end{minipage}
    \begin{center}
    \caption{\label{figure:support_filter2d}
        The comparison of the support regions for the 2D filtering convolution for the point $(x,y)$ with the original (left) and the compact (right) SIAC filtering, for $k = 1, 2, 3$ case.  One can see that the compact SIAC filtering involves much smaller convolution region, especially for higher-order situations.
}
     \end{center}
\end{figure}

\begin{table}[!ht]
\centering 
\begin{tabular}{@{}ccccccccccc@{}}\toprule
  & & & \multicolumn{2}{c}{DG solution} && \multicolumn{2}{c}{Original SIAC} && \multicolumn{2}{c}{Compact SIAC} \\ 
\cmidrule{4-5} \cmidrule{7-8} \cmidrule{10-11}
Degree & Elements & & Error & Order && Error & Order && Error & Order \\ 
\midrule
        & $10\times10$ && 3.71e-02 & --   && 3.20e-02 &  --   && 3.09e-02 &  -- \\
$k = 1$ & $20\times20$ && 7.07e-03 & 2.39 && 4.02e-03 & 2.99  && 3.95e-03 & 2.97\\ 
        & $40\times40$ && 1.57e-03 & 2.17 && 4.97e-04 & 3.01  && 4.93e-04 & 3.00\\ 
        & $80\times80$ && 3.80e-04 & 2.05 && 6.17e-05 & 3.01  && 6.15e-05 & 3.00\\ 
\midrule
        & $10\times10$ && 1.21e-03 &  --  && 3.88e-04 &  --   && 1.67e-04 &  -- \\ 
$k = 2$ & $20\times20$ && 1.51e-04 & 3.00 && 8.30e-06 & 5.55  && 4.62e-06 & 5.18\\ 
        & $40\times40$ && 1.89e-05 & 3.00 && 2.00e-07 & 5.38  && 1.41e-07 & 5.03\\ 
        & $80\times80$ && 2.36e-06 & 3.00 && 6.23e-09 & 5.00  && 5.31e-09 & 4.73\\ 
\midrule
        & $10\times10$ && 4.65e-05 &  --  && 3.25e-05 &  --   && 2.35e-06 &  -- \\
$k = 3$ & $20\times20$ && 2.92e-06 & 3.99 && 1.40e-07 & 7.86  && 1.05e-08 & 7.81\\ 
        & $40\times40$ && 1.83e-07 & 4.00 && 5.76e-10 & 7.92  && 5.99e-11 & 7.45\\ 
        & $80\times80$ && 1.14e-08 & 4.00 && 3.71e-12 & 7.28  && 3.39e-13 & 7.47\\ 
\bottomrule
\end{tabular}
\caption{\label{table:filtered2d}Convergence tests for a 2D advection equation~\eqref{eq:advection2d} for the DG solutions, the filtered solutions with the original, and the compact SIAC filtering.  Here, the final time $T = 2\pi$.}
\end{table}

\begin{figure}[!ht]
    \centering
    \begin{minipage}[c]{0.32\textwidth}
        \centerline{DG}
    \end{minipage}
    \begin{minipage}[c]{0.32\textwidth}
        \centerline{$\,\,$Original SIAC}
    \end{minipage}
    \begin{minipage}[c]{0.32\textwidth}
        \centerline{$\,\,$Compact SIAC}
    \end{minipage}

    \begin{minipage}[c]{0.32\textwidth}
        \includegraphics[width=1.\textwidth]{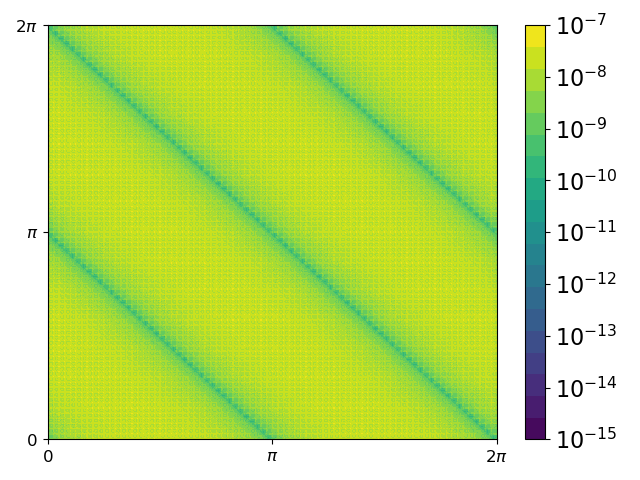}
    \end{minipage}
    \begin{minipage}[c]{0.32\textwidth}
        \includegraphics[width=1.\textwidth]{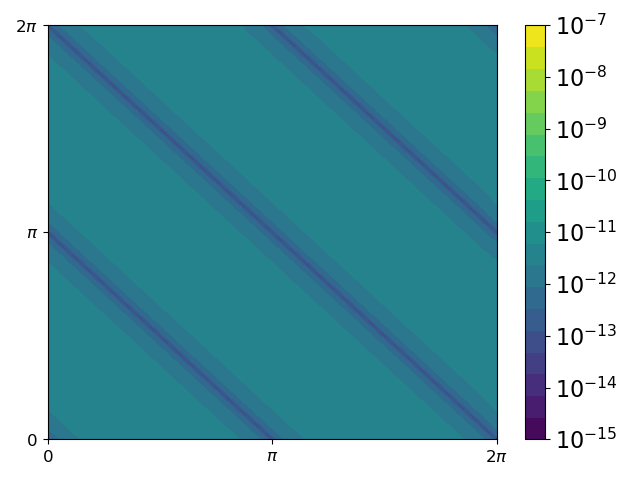}
    \end{minipage}
    \begin{minipage}[c]{0.32\textwidth}
        \includegraphics[width=1.\textwidth]{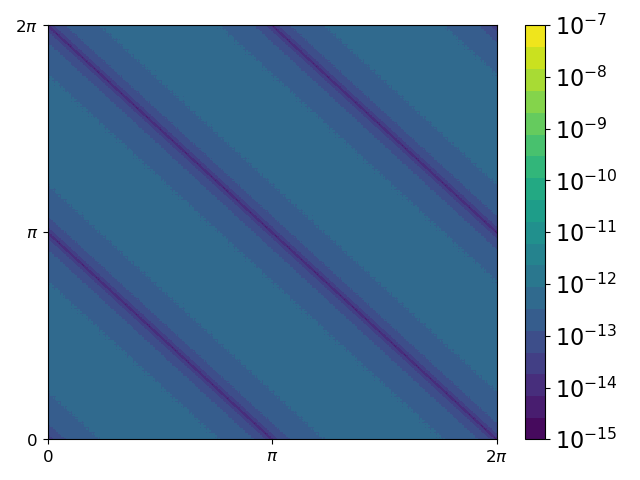}
    \end{minipage}
       
    \caption{\label{fig:filtered2d}
      Contour error plots of a 2D advection equation~\eqref{eq:advection2d} for the DG solution, and the filtered solutions with the original and the compact SIAC filtering for $k=3$ over a uniform $80\times80$ rectangular mesh.}
\end{figure}

\section{Conclusion}
\label{sec:conclusion}
This article has discussed two essential components of constructing a generic accuracy-enhancing filter for the DG method: the choice of basis functions and the distribution of basis functions.  We first used a Fourier transformation to design a general framework to construct a series of basis functions for the SIAC filter with arbitrary initial basis function.  The introduction of this general construction idea will pave the way for the design of new SIAC filters with proven superconvergence properties.   An interesting future research direction is investigating the utility of the construction idea for DG solutions with more complicated features (e.g., DG solutions with shocks or discontinuities, on polygon meshes).   Secondly, by investigating the distribution of the basis functions of the SIAC filter, we propose the idea of the compact SIAC filter.  The compact SIAC filter has dramatically reduced the support size of the original SIAC filter.  We also demonstrate the extension of this idea to other variations of the SIAC filters (e.g., position-dependent SIAC filters).  We also proved that the superconvergent extraction capabilities of the
SIAC filters are unaffected.   We also present numerical results to confirm the theoretical conclusion and demonstrate the better numerical performance of the findings.

\section*{Acknowledgements}
The author was partially supported by the National Natural Science Foundation of China (NSFC) under grants No. 11801062.

\bibliography{ref}
\bibliographystyle{plain}
\end{document}